\documentclass[12pt,reqno]{amsart}

\DeclareMathOperator{\ann}{ann}%

\usepackage{amssymb}
\usepackage{hyperref}
\usepackage[all]{xypic} 
\usepackage{color}
\usepackage{xr}

\newcommand{\F}{\mathbb{F}}

\newcommand{\Fp}{\F_p}

\newcommand{\Gal}{\text{\rm Gal}}

\newcommand{\N}{\mathbb{N}}

\newcommand{\Z}{\mathbb{Z}}
\newcommand{\comment}[1]{}

\makeindex

\begin{document}


\title[Arithmetic encoded in the module structure of $K^{\times}/K^{\times p^m}$]{Arithmetic properties encoded in the Galois module structure of $K^\times/K^{\times p^m}$}

\thanks{The first author is partially supported by the Natural Sciences and Engineering Research Council of Canada grant R0370A01.  He also gratefully acknowledges the Faculty of Science Distinguished Research Professorship, Western Science, in years 2004/2005 and 2020/2021. The second author is partially supported by 2017--2019 Wellesley College Faculty Awards. The third author was supported  in part by National Security Agency grant MDA904-02-1-0061.}

\author[J\'{a}n Min\'{a}\v{c}]{J\'{a}n Min\'{a}\v{c}}
\address{Department of Mathematics, Western University, London, Ontario, Canada N6A 5B7}
\email{minac@uwo.ca}

\author[Andrew Schultz]{Andrew Schultz}
\address{Department of Mathematics, Wellesley College, 106 Central Street, Wellesley, MA \ 02481 \ USA}
\email{andrew.c.schultz@gmail.com}

\author[John Swallow]{John Swallow}
\address{Office of the President, Carthage College, 2001 Alford Park Drive, Kenosha, WI \ 53140 \ USA}
\email{jswallow@carthage.edu}

\begin{abstract}
The power classes of a field are well-known for their ability to parameterize elementary $p$-abelian Galois extensions.  These classical objects have recently been reexamined through the lens of their Galois module structure.  Module decompositions have been computed in several cases, providing deep new insight into absolute Galois groups.  The surprising result in each case is that there are far fewer isomorphism types of indecomposables than one would expect generically, with summands predominately free over associated quotient rings.  Though non-free summands are the exception both in their form and prevalence, they play the critical role in controlling arithmetic conditions in the field which allow the rest of the decomposition to be so simple.  

Suppose $m,n \in \mathbb{N}$ and $p$ is prime.  In a recent paper, a surprising and elegant decomposition for $p^m$th power classes has been computed when the underlying Galois group is a cyclic group of order $p^n$.  As with previous module decompositions, at most one non-free summand appears.  Outside of a particular special case when $p=2$, the structure of this exceptional summand was determined by a vector $\mathbf{a}\in \{-\infty,0,\dots,n\}^m$ and a natural number $d$.  In this paper we give field-theoretic interpretations for $\mathbf{a}$ and $d$, showing they are related to the solvability of a family of Galois embedding problems and the cyclotomic character associated to $K/F$.
%
%
\end{abstract}

\date{\today}

\maketitle

\newtheorem*{theorem*}{Theorem}
\newtheorem*{lemma*}{Lemma}
\newtheorem{theorem}{Theorem}
\newtheorem{proposition}{Proposition}[section]
\newtheorem{corollary}[proposition]{Corollary}
\newtheorem{lemma}[proposition]{Lemma}
\numberwithin{equation}{section}

\theoremstyle{definition}
\newtheorem*{definition*}{Definition}
\newtheorem*{remark*}{Remark}

\parskip=10pt plus 2pt minus 2pt


\section{Introduction}

\subsection{Motivation}

The foundational problem in modern Galois theory is the \emph{inverse Galois problem}: for a given field $F$, what groups appear as the Galois group of some extension of $F$?  If $\widetilde{G}$ is a group, we say that $\widetilde{G}$ is realizable over $F$ if there exists some Galois extension $L/F$ so that $\Gal(L/F) \simeq \widetilde{G}$.  Such an extension $L/F$ is sometimes called a $\widetilde{G}$-extension of $F$. The inverse Galois problem for $F$, therefore, asks for a way to distinguish those groups which are realizable over $F$ from those which are not.  This problem is extraordinarily difficult in general.  For example, the inverse Galois problem for $\mathbb{Q}$ is the subject of intense interest.

There are a number of methodologies one could use when trying to realize a group $\widetilde{G}$ over $F$, but perhaps one of the more intuitive techniques is to review what the fundamental theorem of Galois theory gives us to work with when we already have a $\widetilde{G}$-extension of $F$ in hand.  So suppose $L/F$ is an extension with $\widetilde{G} =  \Gal(L/F)$, and suppose further that $N$ is a subgroup of $\widetilde{G}$.  By Galois theory we know that $N$ corresponds to an intermediate field $K$, and that $\Gal(L/K) = N$.  If $N$ is additionally a normal subgroup, then the extension $K/F$ is Galois, and indeed we get that $\Gal(K/F) \simeq \Gal(L/F)/\Gal(L/K)$ in a natural way.  With this in mind, one might attempt to build a $\widetilde{G}$-extension of $F$ iteratively.  First one identifies a quotient $G$ of $\widetilde{G}$ (coming from a normal subgroup $N$) and finds (or is given) an extension $K/F$ with $\Gal(K/F) \simeq G$.  Then one finds an extension $L/K$ with $\Gal(L/K) \simeq N$.  The hope is that the overall extension $L/F$ provides a suitable ``gluing" of the quotient group and factor group so that $\Gal(L/F) \simeq \widetilde{G}$.  Of course, the idea sketched here can easily run into problems. Notably, the extension $L/F$ might fail to be Galois!  Even if it is Galois, however, the group $\Gal(L/F)$ might be some extension of $N$ by $G$ which is different from $\widetilde{G}$.

This general approach motivates the definition of the Galois embedding problem.  For a field $F$ and groups $\widetilde{G}$ and $G$, suppose we have an extension $K/F$ and an isomorphism $\psi:\Gal(K/F) \xrightarrow{\sim} G$, as well as a surjection $\pi: \widetilde{G} \twoheadrightarrow G$.  The Galois embedding problem for $(\widetilde{G},\psi,\pi)$ over $K/F$ asks whether there is a Galois extension $L/F$ containing $K$ and an isomorphism $\Psi: \Gal(L/F) \xrightarrow{\sim} \widetilde{G}$ so that the natural surjection from Galois theory makes the following diagram commute:
$$\xymatrix{
\Gal(L/F) \ar@{->>}[r] \ar[d]^{\Psi}& \Gal(K/F)\ar[d]^{\psi}\\
\widetilde{G} \ar@{->>}[r] & G
}.$$
For the sake of convenience, we will often refer to the embedding problem $(\widetilde{G},\psi,\pi)$ over $K/F$ as the embedding problem $\widetilde{G} \twoheadrightarrow G$ over $K/F$.  Observe that this is a generalization of the search for $\widetilde{G}$-extensions of $F$, since a $\widetilde{G}$-extension of $F$ is precisely a solution to the embedding problem $\widetilde{G} \twoheadrightarrow \{\text{id}\}$ over $F/F$.

The benefit of studying Galois embedding problems is that they give a natural way to look for complicated Galois extensions of a field using relatively simple building blocks. For example, for a given prime $p$, the $\mathbb{Z}/p\mathbb{Z}$-extensions of a field $K$ have been carefully studied, and in all cases there is a known family $J(K)$ which parameterizes all such extensions.  The most familiar situation is when $K$ contains a primitive $p$th root of unity (so that, in particular, $K$ is not characteristic $p$). In this case, Kummer theory tells us that the $1$-dimensional $\mathbb{F}_p$-subspaces of $J(K):=K^\times/K^{\times p}$ are in correspondence with the $\mathbb{Z}/p\mathbb{Z}$-extensions of $K$.  More generally, elementary $p$-abelian extensions of rank $k$ --- i.e., those for whom the Galois group over $K$ is isomorphic to $\mathbb{Z}/p\mathbb{Z}^{\oplus k}$ --- correspond to $k$-dimensional subspaces of $J(K)$. Since the Galois-theoretic appearance of this family of groups is so well understood, it therefore makes sense to look for embedding problems that involve them. 

Indeed, such embedding problems have been carefully studied.  Continue to assume that $K/F$ is a Galois extension with $\Gal(K/F) \simeq G$, and that we have a surjection $\pi:\widetilde{G} \twoheadrightarrow G$.  Suppose further that $F$ has a primitive $p$th root of unity, and that $N = \ker(\pi)$ is a subgroup of order $p$ within $Z(\widetilde{G})$.  Finally, assume that $\widetilde{G}$ is not the split extension of $G$ by $N$.  Then the solvability of the embedding problem $\widetilde{G} \twoheadrightarrow G$ over $K/F$ is encoded in the value of a particular element in the second cohomology group of $G$ with coefficients in $K^\times$.  More specifically, since $\widetilde{G}$ is an extension of $G$ by $N$, it is represented by some cohomology class $c \in H^2(G,N)$.  By identifying $N$ with the group of $p$th roots of unity $\mu_p \in K^\times$ (which have identical $G$-actions since $\mu_p \subseteq F$) and using the map $H^2(G,\mu_p) \to H^2(G,K^\times)$ induced by inclusion, one then gets a class in $H^2(G,K^\times)$.  The solvability of $\widetilde{G} \twoheadrightarrow G$ over $K/F$ is detected by the triviality of this element.  See \cite{Ledet}.  

This work can be traced back to Dedekind's investigations into the embedding problem $Q_8 \twoheadrightarrow \mathbb{Z}/2\mathbb{Z} \oplus \mathbb{Z}/2\mathbb{Z}$ from \cite{Dedekind}. The interested reader can consult \cite{DameyMartinet,DameyPayan,GrundmanSmith1,GrundmanSmith2,GrundmanSmith3,GrundmanSmithSwallow,Michailov2,Michailov4,Michailov5} for the use of embedding problems to study the appearance of small $2$-groups as Galois groups, as well as \cite{MassyNguyen,Michailov1,Michailov3,Swallow} for analogous results for small $p$-groups when $p$ is odd.

In addition to cohomological techniques, the solvability of certain embedding problems can be interpreted in terms of arithmetic conditions on the base extension.  In \cite{Al}, for example, Albert proved that if $F$ contains a primitive $p$th root of unity $\xi_p$ and $\Gal(K/F) \simeq \mathbb{Z}/p^n\mathbb{Z}$, then the embedding problem $\mathbb{Z}/p^{n+1}\mathbb{Z} \twoheadrightarrow \mathbb{Z}/p^n\mathbb{Z}$ over $K/F$ is solvable if and only if $\xi_p \in N_{K/F}(K^\times)$.  More generally, a similar criterion exists to determine the solvability of the embedding problem $\mathbb{Z}/p^{n+i}\mathbb{Z} \twoheadrightarrow \mathbb{Z}/p^n\mathbb{Z}$ over $K/F$ for $i \geq 1$, at least in the presence of appropriate roots of unity.  Specifically, if $F$ contains a primitive $p^i$th root of unity $\xi_{p^i}$, then the embedding problem $\mathbb{Z}/p^{n+i}\mathbb{Z} \twoheadrightarrow \mathbb{Z}/p^n\mathbb{Z}$ over $K/F$ is solvable if and only if $\xi_{p^i} \in N_{K/F}(K^\times)$. An excellent discussion of this result (using tools from cohomology) --- as well as a number of other interesting topics concerning these types of embedding problems --- is presented in \cite{ArasonFeinSchacherSonn}.   For another excellent reference on embedding problems of this particular form, see \cite{FeinSaltmanSchacher}

It is worth observing that the study of embedding problems of the form $\mathbb{Z}/p^{n+i}\mathbb{Z} \twoheadrightarrow \mathbb{Z}/p^n\mathbb{Z}$ harkens to the investigation of $\mathbb{Z}_p$-extensions of a given ground field.  The interested reader can consult \cite{BertrandiasPayan,Whaples} for more on these types of extensions.  Naturally, this also brings to mind the subject of Iwasawa theory.  An excellent overview of Iwasawa theory is available in \cite{SharifiNotices}, though the reader could consult \cite{CoatesFukayaKatoSujatha,CoatesSchneiderSujatha,CoatesSujathaWintenberger,IwasawaAnnals,Kato,SujathaEulerPoincare} for a sampling of both the history and evolution of this wide-ranging and powerful topic.


Recently, additional strategies have been employed in solving embedding problems.  Massey products have been shown to have strong connections to the solvability of particular embedding problems (see, \cite{MT-Advances,MT-JAlg}), as well as to properties of absolute Galois groups (see \cite{EfratMatzri,MPQT,MT-JEMS}).  More information about Massey products and their connection to Galois theory can be found in \cite{HarpazWittenberg,LamLiuSharifiWakeWang,SharifiCrelle}.

Another recent approach to studying embedding problems comes from Galois modules, particularly when the embedding problem concerns a (group-theoretic) extension of a cyclic $p$-group by an elementary $p$-abelian group.  Let $K/F$ be an extension with $\Gal(K/F) \simeq \mathbb{Z}/p^n\mathbb{Z} =: G$, and let $J(K)$ be the $\mathbb{F}_p$-space whose $k$-dimensional subspaces parameterize rank $k$ elementary $p$-abelian extensions of $K$.  (Again, if $\xi_p \in K$ then this means $J(K) = K^\times/K^{\times p}$.) If $L/K$ is an elementary $p$-abelian extension corresponding to some $\mathbb{F}_p$-subspace $M \subseteq J(K)$, then it is easy to show that $L/F$ is Galois if and only if $M$ is an $\mathbb{F}_p[G]$-submodule.  In fact, when $M$ is a cyclic $\mathbb{F}_p[G]$-submodule generated by some element $m \in J(K)$, Waterhouse showed in \cite{Waterhouse} how to compute the structure of $\Gal(L/F)$ in terms of $\dim_{\mathbb{F}_p}(M)$ and a field-theoretic invariant --- the so-called \emph{index} --- associated to $m$.  This methodology was extended to non-cyclic modules in \cite{Schultz}, providing a Galois-theoretic interpretation for the solvability of any embedding problem which arises as an extension of a cyclic $p$-group by an elementary $p$-abelian group.  

To unlock the power of the dictionary between embedding problems of this type and Galois modules, one would like a decomposition of $J(K)$ into indecomposable summands which is cognizant of the aforementioned index.  Some (purely module-theoretic) results for local field extensions were already available from Borevi\v{c} and Fadeev in the 1960s (see \cite{Borevic,Faddeev}).  Assuming $\xi_p \in K$, the decomposition for general fields was first determined in the case that $n=1$ (i.e., $G = \mathbb{Z}/p\mathbb{Z}$) by the first and third authors in 2003 (\cite{MS}), and for $n \in \mathbb{N}$ by all three authors in 2006 \cite{MSS}. Decompositions for $J(K)$ when $\xi_p \not\in K$ were determined in \cite{BergSchultz,MSS,Schultz}.  The surprise in the decomposition of these modules is that far fewer isomorphism classes of indecomposable summands appear than one expects for a ``generic" $\mathbb{F}_p[G]$-module.   Specifically, let $G_i$ be the quotient of $G$ of order $p^i$, and let $K_i$ be the intermediate extension of $K/F$ to which $G_i$ corresponds.  Then the $\mathbb{F}_p[G]$-module $J(K)$ decomposes as $$J(K) \simeq X \oplus Y_0 \oplus Y_1 \oplus \cdots \oplus Y_n,$$ where for each $0 \leq i \leq n$ we have that $Y_i$ is a free $\mathbb{F}_p[G_i]$-module, and --- outside one special case\footnote{The $X$ summand is trivial only when $p=2, n=1, \text{char}(K) \neq 2$, and $-1\not\in N_{K/F}(K^\times)$.} --- the summand $X$ (the so-called ``exceptional summand") is a cyclic module of dimension $p^{i(K/F)}+1$ for some $i(K/F) \in \{-\infty,0,1,\cdots,n-1\}$. Whereas a ``generic" $\mathbb{F}_p[G]$-module can have up to $p^n$ many isomorphism types of indecomposable summands, for a given extension $K/F$ we instead see that $J(K)$ has at most $n+2$ isomorphism types of summands.  Across all extensions $K/F$, the total number of isomorphism types of summands which can appear in the decomposition of $J(K)$ is $2n+1$.  

The quantity $i(K/F)$ which determines the dimension of the exceptional summand can be interpreted in a variety of ways.  In particular, when $\xi_p \in K$, it is shown in \cite[Th.~3]{MSS} that $i(K/F) = -\infty$ precisely when $\xi_p \in N_{K/F}(K^\times)$, and that otherwise $i(K/F)$ is the minimal value $i \in \{0,1,\cdots,n-1\}$ so that $\xi_p \in N_{K/K_{i+1}}(K^\times)$.  In light of Albert's characterization of embedding problems discussed above, this is equivalent to saying that $i(K/F) = -\infty$ precisely when there is a solution to the embedding problem $\mathbb{Z}/p^{n+1}\mathbb{Z} \twoheadrightarrow \mathbb{Z}/p^n\mathbb{Z}$ over $K/F$, and otherwise that $i(K/F)$ is the minimal value  $i \in \{0,1,\cdots,n-1\}$ for which the embedding problem $\mathbb{Z}/p^{n-i}\mathbb{Z} \twoheadrightarrow \mathbb{Z}/p^{n-i-1}\mathbb{Z}$ is solvable over $K/K_{i+1}$.   It is quite remarkable that this family of cyclic, $p$-power embedding problems along the tower $K/F$ leaves such a significant fingerprint on the overall  structure of this foundational module $J(K)$.

A wealth of information has been gleaned by combining the computed module decomposition with the dictionary between module-theoretic and Galois-theoretic information.  These include results that allow for enumerations of extensions of a particular type (so-called ``realization multiplicity" results), as well as surprising connections between the appearance of groups as Galois groups over a given field (so-called ``automatic realization" results).  The interested reader can consult \cite{BergSchultz,CMSHp3,Schultz} for the former and \cite{CMSHp3,MSS.auto,MS2,Schultz} for the latter.  (Some broader reading on realization multiplicity results is available in \cite{Jensen2,JensenPrestel1,JensenPrestel2}.  If the reader is interested in broader overview of automatic realizations, they might consult \cite{Jensen1,Jensen2,Jensen3}; the connection to embedding problems is explored in \cite{JensenLedetYui,Ledet}.)

We should also mention that there is a related study of additive Galois modules motivated by Emmy Noether's classical work on normal bases (see \cite{Noether-normalbasis}, also found in \cite{Noether-collected}).  This work repeatedly features in the study of integral Galois module structures.  A few representative sources related to this work can be found in \cite{ByottChildsElder,ByottElder,Childs}.  When we compare this additive analysis with our multiplicative one, a number of techniques are clearly similar.  This relationship is most conspicuous when one has exponential and logarithmic functions available, in which case these structures are explicitly interwoven with each other. Although these explorations are different, some future interaction between these two topics could be fruitful.

Motivated by the fascinating structural properties and deep applications of the Galois modules already computed, there has been considerable work to extend the reach of this methodology.  If we continue to assume that the base extension $K/F$ satisfies $\Gal(K/F) \simeq \mathbb{Z}/p^n\mathbb{Z}$ for some $n \in \mathbb{N}$, there have been several investigations into the Galois module structure of Galois cohomology (see \cite{LMSS,LMS}) and Milnor $K$-groups (see \cite{BLMS}).  More recently, work has begun in cases where the base extension $K/F$ is not a cyclic $p$-group.  In \cite{CMSS} the structure of square power classes is considered when $\Gal(K/F) \simeq \mathbb{Z}/2\mathbb{Z} \oplus \mathbb{Z}/2\mathbb{Z}$ and $\text{char}(K) \neq 2$.  In a forthcoming paper \cite{HMS}, the module structure of $J(K)$ is determined when $\Gal(K/F)$ is any finite $p$-group, under the assumption that the maximal pro-$p$ quotient of the absolute Galois group of $K$ is a free pro-$p$-group.  As with the Galois module decompositions described above, in each of these cases the module decompositions are staggering in their simplicity.  Indeed, in these cases the results are even more spectacular since the modular representation theory is so much more complex than in the case where $\Gal(K/F)$ is a cyclic $p$-group.

\subsection{The main results of this paper}

Yet another recent paper explores the Galois module structure of a more refined version of the parameterizing space of elementary $p$-abelian extensions.  Let $\text{char}(K) \neq p$, and again assume that $\Gal(K/F) \simeq \mathbb{Z}/p^n\mathbb{Z} = \langle \sigma \rangle:=G$ for some $n \in \mathbb{N}$.  As before, for $0 \leq i \leq n$ we will write $G_i$ for the quotient of $G$ of order $p^i$, and let $K_i$ be the intermediate field of $K/F$ to which $G_i$ corresponds.  

In \cite{MSS2b} the three authors compute the Galois module structure of $J_m:=K^\times/K^{\times p^m}$.  Again, the structure is far simpler than one would expect for a generic $\mathbb{Z}/p^m\mathbb{Z}[\mathbb{Z}/p^n\mathbb{Z}]$-module, with all but at most one summand isomorphic to $\mathbb{Z}/p^m\mathbb{Z}[G_i]$ for some $0 \leq i \leq n$.  Assuming that $K$ contains a primitive $p$th root of unity (and some additional mild hypotheses in the case $p=2$), such an ``exceptional summand" exists.  Its isomorphism type is determined by a vector $\mathbf{a} \in \{-\infty,0,\cdots,n\}^m$ and a natural number $d \in \mathbb{N}$.  In this paper we look for field-theoretic explanations for the quantities $\mathbf{a}$ and $d$ from this decomposition, and we will determine the degree to which they are unique. 

Both $\mathbf{a}$ and $d$ arise by investigating so-called \emph{minimal norm pairs}.  For $\mathbf{a} = (a_0,a_1,\cdots,a_{m-1}) \in \{-\infty,0,1,\cdots,n\}^m$ with $a_0<n$ and $d \in 1+p\Z$, we say that $(\mathbf{a},d)$ is a \emph{norm pair of length $m$} if there exist $\alpha,\delta_m \in K^\times$ and $\delta_i \in K_{a_i}^\times$  for each $0 \leq i <m$ which satisfy
\begin{equation}\label{eq:defining.properties.for.exceptionalityIII}
    \begin{split}
        \alpha^{\sigma} &= \alpha^{d} \delta_0 \delta_1^p
        \cdots\delta_{m}^{p^m}\\
        \xi_p &= N_{K/F}(\alpha)^{\frac{d-1}{p}}\cdot
        N_{K_{n-1}/F}(\delta_0) \cdot \left(\prod_{i=1}^{m}
        N_{K/F}(\delta_i)^{p^{i-1}} \right).
    \end{split}
\end{equation}
The quantity $\mathbf{a}$ is called the \emph{norm vector} and the integer $d$ is called the \emph{twist}, and one says that $(\alpha,\{\delta_i\}_{i=0}^{m})$ \emph{represents} $(\mathbf{a},d)$.

One can give a partial ordering on the set of norm pairs of length $m$ by defining $(\mathbf{a},d) \leq (\mathbf{a}',d')$ if either $\mathbf{a} < \mathbf{a}'$ (using the natural lexicographical ordering on vectors), or if $\mathbf{a} = \mathbf{a}'$ and $$\min\{v_p(d'-1),m\} \leq \min\{v_p(d-1),m\}$$ (where $v_p(x)$ denotes the exponent of $p$ in the prime factorization of $x$).  Observe that if $d \equiv \tilde d \pmod{p^m}$, then $(\mathbf{a},d) \leq (\mathbf{a},\tilde d) \leq (\mathbf{a},d)$.



Propositions 7.1 and 7.2 of \cite{MSS2b} assert that if $(\mathbf{a},d)$ is chosen minimally with respect to this ordering (a so-called \emph{minimal norm pair}), and if $(\alpha,\{\delta_i\}_{i=0}^{m})$ represents $(\mathbf{a},d)$, then the ``exceptional summand" in a decomposition of $J_m$ can be chosen as the module generated by the classes in $J_m$ represented by $\alpha, \delta_1,\dots,$ and $\delta_{m-1}$.  Moreover, the relations between the generators of this summand are generated by 
\begin{equation}\label{eq:definition.of.Xadm}
\begin{split}
\alpha^{\sigma-d} &= \prod_{i=0}^{m-1}  \delta_i^{p^i} \pmod{K^{\times p^m}}\\ 
\delta_i^{\sigma^{p^{a_i}}} &= \delta_i \pmod{K^{\times p^m}} \text{ for all }0 \leq i < m.
\end{split}
\end{equation}  (Here and throughout the paper, we take $p^{-\infty}=0$ and $\sigma^{p^{-\infty}}=0$.  Hence if $a_i = -\infty$, then $\delta_i=1$.
)  

It turns out that interpretations for the quantities $\mathbf{a}$ and $d$ depend in significant ways on the $p$-power roots of unity that appear along the tower $K/F$.  For $k \in \mathbb{N}$, let $\mu_{p^k}$ denote the set of $p^k$th power roots of unity within some algebraic closure of $K$.  We will let $\xi_{p^k}$ denote a fixed generator of (the multiplicative group) $\mu_{p^k}$.  If $\mu_{p^k} \subseteq F$ for all $k \in \mathbb{N}$, set $\omega=\infty$.  Otherwise, let $\omega$ be chosen so that $\mu_{p^\omega} \subseteq F$, but $\mu_{p^{\omega+1}} \not\subseteq F$.  In the same way, set $\nu = \infty$ if $\mu_{p^k} \subseteq K$ for all $k \in \mathbb{N}$, and otherwise define $\nu$ so that $\mu_{p^\nu} \subseteq K$ but $\mu_{p^{\nu+1}}\not\subseteq K$.   Because the case $p=2$ and $\omega = 1<\nu$ presents its own idiosyncratic technicalities, we will exclude this case from our investigation. One benefit of excluding this case is that the ``bottom" of the extension $K/F$ can be described as a cyclotomic extension in an easy way: for any $\omega \leq k \leq \nu$ we get $F(\xi_{p^k}) = K_{k-\omega}$.

Our main theorems are as follows.  The first concerns an interpretation of the quantity $d$ from the decomposition.  

\begin{theorem}\label{th:d.is.cyclotomic.character}
Suppose that $\xi_p\in F$.  If $p = 2$ and $n=1$ suppose that $-1 \in N_{K/F}(K^\times)$ as well. Furthermore, assume that if $p=2$ then we are not in the case $\omega =1 < \nu$.  Let $(\mathbf{a},d)$ be a minimal norm pair.
\begin{enumerate}
\item\label{it:d.is.cyclo.character...any.d.works} If $\omega<\nu$ then for any $\tilde d$ satisfying $\tilde d = d \pmod{p^\nu}$, the pair $(\mathbf{a},\tilde d)$ is a minimal norm pair.
\item\label{it:d.is.cyclo.character...omega.equals.nu} If $\omega = \nu$, then $d \equiv 1 \pmod{p^m}$.
\item\label{it:d.is.cyclo.character...d.is.cyclo.character} $d$ is equivalent modulo $p^{\min\{m,\nu\}}$ to the cyclotomic character associated to the extension $K/F$.
\end{enumerate}
\end{theorem}

The second result gives an arithmetic interpretation for the vector $\mathbf{a}$ from the decomposition.

\begin{theorem}\label{th:a.in.relation.to.norms}
Continue with the hypotheses of Theorem \ref{th:d.is.cyclotomic.character}.  
For $0\le i<\nu$ set
        \begin{equation*}
            b_i =
            \begin{cases}
                -\infty, &\xi_{p^{i+1}}\in
                N_{K/F(\xi_{p^{i+1}})}(K^\times) \\
                \min \{s : \xi_{p^{i+1}}\in
                N_{K/K_{s+1}}(K^\times)\}, &\text{otherwise}
            \end{cases}
        \end{equation*}
        and set $b_\nu=n$.  Then $a_0=b_0$ and for $0 < i<m$,
        \begin{equation*}
            a_i =
            \begin{cases}
                b_i, & b_i> b_{i-1}+1 \text{\ and\ } i\le\nu\\
                -\infty, & b_i= b_{i-1} +1 \text{\ or\ } i>\nu.
            \end{cases}
        \end{equation*}
        In particular, $\mathbf{a} = (-\infty,\cdots,-\infty)$ if either $\nu = \infty$ or $K/F$ is cyclotomic.
\end{theorem}

Before we consider interpretations for this result, we make two remarks about well-definition issues that it presents.

\begin{remark*}
Suppose that $\omega \leq \ell < \nu$.  Since we assume that $\omega>1$ when $p=2$, it follows that $F(\xi_{p^{\ell+1}}) = K_{\ell+1-\omega}$.  In particular, we have $\xi_{p^{\ell+1}} \not\in K_{\ell-\omega}$, and therefore $\min\{s: \xi_{p^{\ell+1}} \in N_{K/K_{s+1}}(K^\times)\} \geq \ell-\omega$.  From this we see that if $b_\ell \neq -\infty$, then we must have $b_\ell \geq \ell-\omega$.  In fact, we claim that the value $b_\ell = \ell-\omega$ is also impossible.  If we have $\ell-\omega \in \{s: \xi_{p^{\ell+1}} \in N_{K/K_{s+1}}(K^\times)\}$, then since $K_{\ell+1-\omega} = F(\xi_{p^{\ell+1}})$, we get $\xi_{p^{\ell+1}} \in N_{K/F(\xi_{p^{\ell+1}})}(K^\times)$.  But then the definition of $b_\ell$ gives us $b_\ell = -\infty$.

Given these bounds on the values of $\mathbf{b} = (b_0,b_1,\cdots,b_{m-1})$ and the relationship between $\mathbf{a}$ and $\mathbf{b}$ described in Theorem \ref{th:a.in.relation.to.norms}, one might want reassurance that for $\omega \leq \ell < \nu$ with $a_\ell \neq -\infty$, we have $a_\ell > \ell-\omega$.  Fortunately, just such a bound is described later in this paper in Proposition \ref{pr:minimality.conditions}(\ref{it:exc.mod.indecom.condition...lower.bound.based.on.UIII}) (using the fact that the value of $t$ from  Proposition \ref{pr:minimality.conditions}(\ref{it:exc.mod.indecom.condition...lower.bound.based.on.UIII}) must be $\omega$ based on Theorem \ref{th:d.is.cyclotomic.character}).
\end{remark*}

\begin{remark*}
In relating $\mathbf{a}$ to $\mathbf{b}$, the statement of Theorem \ref{th:a.in.relation.to.norms} suggests that for each $0 \leq i < \nu$, either one has $b_i=b_{i-1}+1$ or $b_i>b_{i-1}+1$.  In particular, the case $b_i = b_{i-1}>-\infty$ would seem to be excluded.  The fact that this case can be ignored is a consequence of the multiplicative properties of the norm map.  Suppose that we have $b_i = s >-\infty$, so that $\xi_{p^{i+1}} = N_{K/K_{s+1}}(\gamma)$ for some $\gamma \in K^\times$.  Note that since $b_i > -\infty$, we must have $K_{s+1} \supsetneq F(\xi_{p^{i+1}})$.  We will see later in Lemma \ref{le:norm.of.root.of.unity} that $N_{K/K_s}(\gamma) = N_{K_{s+1}/K_s}(\xi_{p^{i+1}}) = \xi_{p^i}$, and hence $b_{i-1} \leq s-1$.
\end{remark*}

One observation to make about Theorem \ref{th:a.in.relation.to.norms} is that the value for $b_0$ is precisely the definition of the quantity $i(K/F)$ that we described earlier from the Galois module decomposition of $K^\times/K^{\times p}$ in the case where $\xi_p \in F$. Indeed, this is no coincidence, for in \cite[Prop.~4.8]{MSS2b} it is established that $i(K/F) = a_0$ when $(\mathbf{a},d)$ is a minimal norm pair.  Given that $i(K/F)$ (and therefore $a_0$) can be phrased in the language of the solvability of a particular class of embedding problems along the tower $K/F$, one might wonder whether the other entries of $\mathbf{a}$ have a  similar characterization.  Happily, the answer is a resounding ``yes," since we have already discussed how the solvability of an embedding problem $\mathbb{Z}/p^{i+n}\mathbb{Z} \twoheadrightarrow \mathbb{Z}/p^n\mathbb{Z}$ for a given extension is measured by the appearance of a primitive $p^i$th root of unity as a norm in that extension.

\begin{corollary}\label{cor:Albert.result.for.higher.embedding.problems}
Continue with the hypotheses of Theorem \ref{th:d.is.cyclotomic.character}, and for $0 \leq i < \nu$ let $b_i$ be defined as in Theorem \ref{th:a.in.relation.to.norms}.  Let $p^I=[K:F(\xi_{p^{i+1}})]$.  Then $b_i = -\infty$ if and only if the embedding problem $\mathbb{Z}/p^{I+i}\mathbb{Z} \twoheadrightarrow \mathbb{Z}/p^{I}\mathbb{Z}$ over $K/F(\xi_{p^{i+1}})$ is solvable.  Otherwise, 
$$b_i=\min\{s: \mathbb{Z}/p^{n-s-1+i}\mathbb{Z} \twoheadrightarrow \mathbb{Z}/p^{n-s-1}\mathbb{Z} \text{ over }K/K_{s+1} \text{ is solvable}\}.$$
\end{corollary}

\subsection{Outline of paper} We have attempted to make this paper as self-contained and approachable as possible.  In particular, we begin the paper with some content that will bring the reader up to speed on the relevant details from the decomposition of $K^\times/K^{\times p^m}$ from \cite{MSS2b}, including some basic arithmetic and module-theoretic facts.  These can be found in section \ref{sec:notation}.  In section \ref{se:np3} we prove Theorem \ref{th:d.is.cyclotomic.character}.  Section \ref{se:more.module.properties} is spent establishing some additional module-theoretic properties of $X$ which are then put to use in section \ref{se:proof1} to prove Theorem \ref{th:a.in.relation.to.norms}.

\subsection{Acknowledgements}

We gratefully acknowledge discussions and collaborations with our friends and colleagues D.~Benson, B.~Brubaker, J.~Carlson, S.~Chebolu, F.~Chemotti, I.~Efrat, A.~Eimer, J.~G\"{a}rtner, S.~Gille, P.~Guillot, L.~Heller, D.~Hoffmann, J.~Labute, T.-Y.~Lam, R.~Sharifi, N.D.~Tan, A.~Topaz, R.~Vakil, K.~Wickelgren, O.~Wittenberg, which have influenced our work in this and related papers. We are also extremely grateful to an anonymous referee for their help in improving the quality and clarity of the manuscript.

\section{Notation and Background}\label{sec:notation}

Throughout the balance of the paper we will adopt the following notation.  We set $G:=\Gal(K/F) = \langle \sigma \rangle \simeq \Z/p^n\Z$, and we write $G_i=\Gal(K_i/F)$ where $K_i$ is the intermediate field of degree $p^i$ over $F$.  We assume that $F$ contains a primitive $p$th root of unity, and that $a_0 = -\infty$ if $p=2$ and $n=1$.  We define $\omega$ and $\nu$ as we did prior to the statement of Theorem \ref{th:d.is.cyclotomic.character}. For $m \in \N$ we define $R_m=\Z/p^m\Z$, so that $J_m$ is naturally an $R_mG$-module.  We define $p^{-\infty}=0$ and $\sigma^{p^{-\infty}}=0$.  The set $1+p^i\Z$ will be denoted $U_i$; of course, there is a natural filtration $U_1 \supseteq U_2 \supseteq U_3 \supseteq \cdots$.  Finally, to make equations easier to read, we move to additive notation when describing the module $J_m$ and its elements.

Throughout this paper we will write $[\gamma]_m$ to indicate the class in $J_m$ represented by $\gamma \in K^\times$.  In a similar way, when $A \subset K^\times$ we will write $[A]_m$ to indicate the image of $A$ within $J_m$. The module $J_m$ will always be written additively, so that the action of $R_mG$ is multiplicative.  We will also adopt additive notation when discussing the group $K^\times$ as well; this will make certain expressions which involve complicated actions of $R_m$ simpler for the reader to read.  In particular we will write $0$ for $1 \in K^\times$ and use $\frac{1}{p}$ to denote the $p$th root of an element.

\subsection{Norm operators}

In studying the structure of $K^{\times}/K^{\times p^m}$, certain polynomials in $\Z G$ play a major role.

\begin{definition*}
For $0\le j\le i \leq n$ and $d \in U_1$, set
\begin{align*}
    P(i,j) 
    &:= \sum_{k=0}^{p^{i-j}-1} \sigma^{kp^j} \in \Z[\sigma]\\
    Q_d(i,j) 
    &:= \sum_{k=0}^{p^{i-j}-1}
    (d^{p^j})^{p^{i-j}-1-k}(\sigma^{p^j})^k \in \Z[\sigma]\\
    T_{d}(i) 
    &:= \sum_{s=1}^{p^i-1} \sum_{k=0}^{s-1}
    d^k\sigma^{s-1-k} \in \Z[\sigma].
\end{align*}
We will often abuse notation and refer to the images of these polynomials in $\Z G$ and $R_m G$ by these same names.  The map $\phi_d:\Z G \to \Z$ is induced by the evaluation $\sigma \mapsto d$ (extended in the natural way to powers of $\sigma$), and when $d^{p^n} =1 \pmod{p^m}$ we write $\phi_d^{(m)}:R_m G \to R_m$ for the induced ring homomorphism.  Finally, we have natural maps $\Psi_j: \Z G \to \Z G_j$ and $\chi_j:R_m G \to R_m G_j$.
\end{definition*}

It is not difficult to show that for $0\le k\le j\le i$, one has
\begin{align*}
    P(i,k) &= P(i,j)P(j,k) \\
    Q_d(i,k) &= Q_d(i,j)Q_d(j,k)\\
    (\sigma^{p^j}-1)P(i,j) &= (\sigma^{p^i}-1) \\
    (\sigma^{p^j}-d^{p^j})Q_d(i,j) &= (\sigma^{p^i}-d^{p^{i}})\\
    (\sigma-d)T_d(i) &= \sum_{s=0}^{p^i-1} (\sigma^s-d^s)\\
\end{align*}

The polynomials $P(i,j)$ and $Q_d(i,j)$ were studied extensively in \cite{MSS2a} and used throughout \cite{MSS2b}; the polynomial $T_d(i)$ will be useful in certain calculations we perform later in this paper.  The following results will be useful to us in our current investigation.

\begin{lemma}[{\cite[Lem.~2.1]{MSS2a}}]\label{le:upowerIII}
    Suppose that $i\in \N$, that $d\in U_i$,
    and that $j\ge 0$.  Then
    \begin{equation*}
        \begin{cases}
            d^{p^j}\in U_{i+j},
            &p>2 \text{\ or\ } i>1 \text{\ or\ } j=0\\
            d^{p^j}=1, &p=2, \ i=1, \ j>0, \ d=-1\\
            d^{p^j}\in U_{v+j},
            &p=2, \ i=1, \ j>0, \ d\in -U_v
        \end{cases}
    \end{equation*}
	When additionally $d \not\in U_{i+1}$, then we have $d^{p^j} \not\in U_{i+j+1}$ in the first case; likewise if $d \not\in -U_{v+1}$ in the last case, then $d^{p^j} \not\in U_{v+j+1}$.
\end{lemma}

\begin{lemma}[{\cite[Lem.~2.2]{MSS2a}}]\label{le:upowerpnot2III}
    Suppose that $p>2$, $d\in U_i$, and $i\ge 1$.  Write $d=1+p^i x$.  Then there exist $f,g \in \Z$ so that
    \begin{equation*}
        d^{p^j} = 1+p^{i+j }x\left(1+fxp^i + gxp^{i+1}\right).
    \end{equation*}


    Now suppose that $p=2$ and $d\in U_i$ for $i\ge 1$.  Write $d=1+2^i x$.  Then there exists $c \in \Z$ so that
    \begin{equation*}
        d^{2^{j}} = 1+2^{i+j}x\left(1+2^{i-1}x + 2^ixc\right).
    \end{equation*}
\end{lemma}

\begin{lemma}[{\cite[Lem.~2.5]{MSS2a}}]\label{le:kerbasicIII}
    Suppose $m\in \N$.
    For $0\le i\le n$ and $0\le k\le m$,
    \begin{equation*}
        \ann_{R_mG_i} p^k = \langle p^{m-k}\rangle.
    \end{equation*}

    For $0\le j<i\le n$ and $0\le k\le m$,
    \begin{equation*}
        \ann_{R_mG_i} p^k(\sigma^{p^j}-1)= \langle P(i,j),
        p^{m-k}\rangle.
    \end{equation*}
\end{lemma}

\begin{lemma}[{\cite[Lem.~2.6]{MSS2a}}]\label{le:phidbIII}
    Suppose $m\in \N$, $0\le i\le n$, and $d\in U_1$.  Then
    \begin{equation*}
        \ann_{R_mG_i} (\sigma-d) = \langle p^kQ_d(i,0)\rangle
    \end{equation*}
    where $k=\min\{v\ge 0 : p^v(d^{p^i}-1)=0 \mod p^m\}$.
\end{lemma}

\begin{lemma}[{\cite[Lem.~2.7]{MSS2a}}]\label{le:qhomoIII}
    Suppose $m\in \N$ and $0\le j< i\le n$, and let $d \in U_1$ be given.  Suppose further that either $p>2$, $d \in U_2$ or $j>0$.   Then $$\chi_j(Q_d(i,0)) =  p^{i-j} Q_d(j,0) \mod p^{i-j+1}R_mG_j.$$
\end{lemma}

\begin{lemma}\label{le:annzsmd}
    Suppose $m\in \N$, $0\le i\le n$, and $d\in U_1$.  Then
    \begin{equation*}
        \ann_{\Z G_i} (\sigma-d) = \begin{cases}
            \langle P(i,0)\rangle, &d=1,\ i>0\\
            \Z, &d=1, \ i=0\\
            \langle Q_{-1}(i,0)\rangle, &d=-1,\ i>0\\
            \langle 0\rangle, &d=-1,\ i=0\\
            \langle 0\rangle, &\text{otherwise}.
        \end{cases}
    \end{equation*}
\end{lemma}

\begin{proof}
    First suppose that $i=0$, so that $\Z G_i\simeq \Z$ and
    we identify $\sigma$ and $1$.  Then when $d=1$,
    $\ann \sigma-d=\ann 0=\Z$, and when $d\neq 1$, $\ann \sigma-d
    =\ann s$ for $s\in \Z\setminus \{0\}$, and so $\ann \sigma-d=\{0\}$.

    Now assume that $i>0$. Observing that $\Z G_i\subset \Z_p G_i$ and $\Z_p G_i$ is the inverse limit of $R_m G_i$ for $m\in \N$, the problem reduces to showing the results for $\Z_p G_i$.  These results follow in the case $d=1$ from Lemma~\ref{le:kerbasicIII}, in the other cases from  Lemmas~\ref{le:phidbIII} and \ref{le:upowerIII}.
\end{proof}

\begin{lemma}\label{le:tdibar}
    For $d\in U_1$ and $0\le j< i\le n$, we have
    \begin{equation*}
        \Psi_j\left(T_d(i)\right) = \begin{cases}
        p^{i-j}T_d(j) + \frac{p^i(p^{i-j}-1)}{2}P(j,0),
        &d=1\\
        2^{i-j-1}T_d(j), &d=-1, j>0\\
        2^{i-j-1}, &d=-1, j=0\\
        p^{i-j}T_d(j)+\left(\sum_{t=0}^{p^{i-j}-1} \frac{d^{tp^j}-1}{d-1}\right)
        Q_d(j,0), &\text{otherwise}.
        \end{cases}
    \end{equation*}
\end{lemma}

\begin{proof}
    To simplify notation, for $1\le v\le p^n$ set
    \begin{equation*}
        q_{d,v} := \sum_{k=0}^{v-1} d^k \sigma^{v-1-k}.
    \end{equation*}
    (We use a subscript for $v$ in place of an argument to emphasize
    that $v$ is not an exponent of $p$.)  Observe that $q_{d,p^i} = Q_d(i,0)$.)

    We first consider the cases $d=1$ and $d=-1$.

    Suppose that $d=1$.  Then
    \begin{equation*}
        T_1(i) = \sum_{s=1}^{p^i-1} q_{1,s}.
    \end{equation*}
    Observe that for all $t$, $\Psi_j(\sum_{s=tp^j}^{(t+1)p^j-1} \sigma^s)
    =P(j,0)=q_{1,p^j}.$  Hence
    \begin{equation*}
        \Psi_j(T_1(i)) = \sum_{s=1}^{p^i-1} q_{1,s\bmod p^j}+ \left[\frac{s}{p^j}\right] P(j,0),
    \end{equation*}
    where $[\cdot]$ denotes the greatest-integer function.
    We conclude
    \begin{align*}
        \Psi_j(T_1(i)) &= \sum_{s=1}^{p^i-1} q_{1,s\bmod p^j}+ \left[\frac{s}{p^j}\right] P(j,0)\\
        &= p^{i-j}T_1(j) + \left(\sum_{s=1}^{p^{i}-1} \left[\frac{s}{p^j}\right]P(j,0)\right)\\
        &= p^{i-j}T_1(j) + \frac{p^jp^{i-j}(p^{i-j}-1)}{2} P(j,0),
    \end{align*}
    as desired.

    Now suppose that $d=-1$.  Since $d\in U_1$ we must have $p=2$.  Then
    \begin{equation*}
        T_{-1}(i) = \sum_{s=1}^{2^i-1} q_{-1,s}.
    \end{equation*}
If $j=0$ then $\Psi_j(q_{-1,s}) = s \bmod 2$ and so $\Psi_0(T_{-1}(i)) = 2^{i-1}$.  Assume then that $1\le j<i$. Observe that for all even $s$, $q_{-1,s}+q_{-1,s+1} = \sigma^{s}$ and therefore $T_{-1}(i) = P(i,1)$.  Then $\Psi_j(P(i,1))=2^{i-j-1}P(j,1) =2^{i-j-1}T_{-1}(j)$, as desired.

    Finally, suppose that $d\neq 1$, $d\neq -1$.  Since $\ann_{\Z G_j} (\sigma-d)=\{0\}$ by Lemma~\ref{le:annzsmd}, we
    check the result by multiplying by $\sigma-d$ and observing equality.
    On one hand, we have
    \begin{align*}
        (\sigma-d)T_d(i) &= \sum_{s=0}^{p^i-1} \sigma^{s}-d^s\\
        &= \sum_{t=0}^{p^{i-j}-1}\sum_{k=0}^{p^j-1} \sigma^{tp^j+k}-
        d^{tp^j+k}
    \end{align*}
    and so
    \begin{align*}
        \Psi_j((\sigma-d)T_d(i))
        &= \sum_{t=0}^{p^{i-j}-1} \sum_{k=0}^{p^j-1}
        \sigma^{k} -d^k + d^k - d^{tp^j+k}\\
        &= p^{i-j}\left( \sum_{k=0}^{p^j-1} \sigma^k - d^k\right)
        + \left(\sum_{k=0}^{p^j -1} d^{k}\right)\left(\sum_{t=0}^{p^{i-j}-1}
        1-d^{tp^j}\right).
    \end{align*}
    Now observe
    \begin{equation*}
        (\sigma-d)p^{i-j}T_d(j) = p^{i-j}\left(\sum_{k=0}^{p^j-1} \sigma^k-d^k\right).
    \end{equation*}
    We also have
    \begin{equation*}
        (\sigma-d)\left(\sum_{t=0}^{p^{i-j}-1} \frac{d^{tp^j}-1}{d-1}\right)
        Q_d(j,0) = \sum_{t=0}^{p^{i-j}-1} \frac{d^{tp^j}-1}{d-1} \left(\sigma^{p^j}-d^{p^j}\right),
    \end{equation*}
    so in $\Z G_j$,
    \begin{align*}
        (\sigma-d)\left(\sum_{t=0}^{p^{i-j}-1} \frac{d^{tp^j}-1}{d-1}\right)
        Q_d(j,0) &= \sum_{t=0}^{p^{i-j}-1} \frac{d^{tp^j}-1}{d-1} \left(1-d^{p^j}\right) \\
        &= \sum_{t=0}^{p^{i-j}-1} (d^{tp^j}-1)\left(\frac{1-d^{p^j}}{d-1}\right) \\
        &= \left(\sum_{t=0}^{p^{i-j}-1} 1-d^{tp^j}\right)\left(\sum_{k=0}^{p^{j}-1} d^k\right),
    \end{align*}
    as desired.
\end{proof}

\subsection{Previously established results on $\mathbf{a}$ and $d$}

In order to be a minimal norm pair, there are certain conditions that $\mathbf{a}$ and $d$ must satisfy. Though these conditions were motivated by field-theoretic concerns (namely, finding a minimal solution to (\ref{eq:defining.properties.for.exceptionalityIII})),  each plays a crucial role in establishing the indecomposability of $X$ in \cite{MSS2a}.  (In fact, there are additional conditions that must be satisfied when $p=2$ and $d \not\in U_2$ to ensure indecomposability; we will see that our assumptions on $\omega$ and $\nu$ when $p=2$ prevent us from falling into this case, and so we haven't included these additional conditions.)

\begin{proposition}\label{pr:minimality.conditions}
Suppose that $(\mathbf{a},d)$ is a minimal norm pair and that $(\alpha,\{\delta_i\}_{i=0}^m)$ represents $(\mathbf{a},d)$.  Suppose furthermore that if $p=2$ then $d \in U_2$. Then
    \begin{enumerate}\addtolength{\itemsep}{5pt}

        \item\label{it:exc.mod.indecom.condition...power.of.dIII} $d^{p^{a_i}}\in U_{i+1}$ for all $0\le i<m$;
        \item\label{it:exc.mod.indecom.condition...ai.vs.aj.inequalityIII}
          $a_i+j<a_{i+j}$ for all $0\le i<i+j<m$ and $a_{i+j}\neq -\infty$;
        \item\label{it:exc.mod.indecom.condition...lower.bound.based.on.UIII} if $d\in U_t\setminus U_{t+1}$ and $t+k<m$, then $a_{t+k} > k$ whenever $a_{t+k}\neq -\infty$; and
		\item\label{it:exc.module.indecom.condition...deltafree} the $\F_pG_{a_i}$-module $\langle [\delta_i]_1 \rangle$ is free for all $0 \leq i <m$ with $a_i \neq -\infty$.
    \end{enumerate}
\end{proposition}
\begin{proof}
(\ref{it:exc.mod.indecom.condition...ai.vs.aj.inequalityIII})-(\ref{it:exc.module.indecom.condition...deltafree}) are proven in (respectively) Proposition~5.2(1), Proposition~5.2(2), and Proposition~6.1 of \cite{MSS2b}.

To prove (\ref{it:exc.mod.indecom.condition...power.of.dIII}), note that if $d\in U_m$, we are done. Otherwise, $d\in U_t\setminus U_{t+1}$ for some $1\le t<m$. If $i<t$ then since $d\in U_t$ we have $d^{p^{a_i}}\in U_t \subset U_{i+1}$.  If $i\ge t$ then (\ref{it:exc.mod.indecom.condition...lower.bound.based.on.UIII}) gives $a_{i}\geq  i-t+1$, and by Lemma~\ref{le:upowerIII} we have $d^{p^{a_i}}\in U_{t+a_i} \subset U_{i+1}$, as desired.
\end{proof}


We also include two results which allow one to use the existence of a certain norm pair to prove the existence of another.  

\begin{lemma}\label{le:contracting.and.d.equivalence}
If $(\mathbf{a},d)$ is a norm pair of length $m$ and $\check d = d\pmod{p^m}$, then $(\mathbf{a},\check d)$ is a norm pair of length $m$ as well.  If $(\mathbf{a},d)$ is a minimal norm pair of length $m$ and $1 \leq s < m$, then $((a_0,\dots,a_{s-1}),d)$ is a minimal norm pair of length $s$.
\end{lemma}

\begin{proof}
The first statement is \cite[Prop.~4.2]{MSS2b}.  The second statement is \cite[Prop.~4.3]{MSS2b}.
\end{proof}

\begin{lemma}[{\cite[Lem.~5.4]{MSS2b}}]\label{le:delicateIII}
    Let $m\ge 2$, and
    suppose  $((a_0,\cdots,a_{m-1}),d)$ is a norm pair
    of length $m$ so that for some $i$ with $0\le i<m-1$ and $a_i\neq -\infty$ we have $0<a_{m-1}-a_i\le m-1-i$.  Finally, suppose that either $p>2$, $a_i>0$, or $d\in U_2$. Then $((a_0,\dots,a_{m-2},-\infty),d)$ is a norm pair.
\end{lemma}

\section{Norm Pairs and the Cyclotomic Character}\label{se:np3}

The main goal of this section is to prove Theorem \ref{th:d.is.cyclotomic.character}.  

\subsection{Two special cases}

We start with two special cases: first when $\nu = \omega = \infty$, and second when $K/F$ is cyclotomic. It will turn out that the results we establish are sufficient to prove Theorem \ref{th:a.in.relation.to.norms} in these cases, though we won't formally do so until section \ref{se:proof1}.

\begin{proposition}\label{pr:nuinfty}
    Assume the hypotheses of Theorem \ref{th:d.is.cyclotomic.character}, and suppose that $\nu=\infty$. 
    Then the minimal norm pair of length $m$ for $K/F$ is $((-\infty,\dots,-\infty),1)$.
\end{proposition}
\begin{proof}[Proof of Proposition~\ref{pr:nuinfty}]
    Since $\nu=\infty$, there exists a primitive $p^i$th root of unity in $K$ for every $i$.  Moreover, since we have assumed that we are not in the case $p=2$ and $\omega=1<\nu$, it must be the case that $\omega = \infty$ as well. By Kummer theory $K=F(\root{p^n}\of{a})$ for some $a\in F^\times$.  Let $\alpha=\root{p^n}\of{a}$.  Then $(\sigma-1)\alpha=\xi_{p^n}$ for some primitive $p^n$th root of unity $\xi_{p^n}\in F^\times$.  Let $\delta_j=0$ for $0\le j<m$, and set $\delta_m=\xi_{p^{n+m}}$, where we choose $\xi_{p^{n+m}}$ such that $p^m \xi_{p^{n+m}} =\xi_{p^n}$.
    Then
    \begin{equation*}
      (\sigma-1)\alpha = p^m\delta_m
    \end{equation*}
    and
    \begin{equation*}
      p^{m-1}(N_{K/F}\delta_m) = p^{m-1}\xi_{p^m} = \xi_p
    \end{equation*}
    for some primitive $p$th root of unity $\xi_p$. Hence
    $(\alpha,\{\delta_i\})$ represents the norm pair
    $((-\infty,\dots,-\infty),1)$.  By minimality, we are done.
\end{proof}

Next we treat the cyclotomic case.  Before doing so, we prove the following

\begin{lemma}\label{le:norm.of.root.of.unity}
Let $\xi_{p^s}$ be a primitive $p^s$th root of unity in $K'$, where $K'/F'$ is an extension with $\Gal(K'/F') \simeq \Z/p\Z$.  If $p=2$ then assume either $s=1$ or that $F'$ contains a primitive $4$th root of unity $\xi_{4}$.  Then $N_{K'/F'}(\xi_{p^s}) = \xi_{p^{s-1}}$ for some primitive $p^{s-1}$st root of unity $\xi_{p^{s-1}}$.
\end{lemma}

\begin{proof}
First, observe that if $\xi_{p^s} \in F'$ then $N_{K'/F'}(\xi_{p^s}) = \xi_{p^s}^p$, and so the result follows.  Therefore suppose that $\xi_{p^s} \in K'\setminus F'$; in particular, if $p=2$ our hypotheses give $s>2$.  Since $K'/F'$ is a $p$-extension (and $\xi_4 \in F'$ if $p=2$), this forces $F'$ to contain all $p^{s-1}$st roots of unity.  If we write $\tau$ for a generator of $\Gal(K'/F')$ then there exists $e \in \Z\setminus\{1\}$ so that $\xi_{p^s}^\tau = \xi_{p^s}^e$.  We then have
$$N_{K'/F'}(\xi_{p^s}) = \left(\sum_{i=0}^{p-1} e^i\right)\xi_{p^s} = \left(\frac{e^p-1}{e-1}\right) \xi_{p^s}.$$ 

Because $p \xi_{p^{s}}\in F'$ we must have $$p\xi_{p^s}= \tau p\xi_{p^s} =p\tau \xi_{p^s} = pe \xi_{p^s},$$ and hence $e \in U_{s-1}$.  Because $\xi_{p^s} \not\in F'$, we must have $e \not\in U_s$.  
By Lemma \ref{le:upowerpnot2III}, when $p>2$ we have $\frac{e^p-1}{e-1} \in p + p^s\Z$, which gives the desired result.  On the other hand, when $p=2$ we have $\frac{e^2-1}{e-1} = e+1 \in 2 + 2^{s-1} + 2^s\Z$, and so 
$$N_{K'/F'}(\xi_{2^s}) = -1 + 2\xi_{2^s}.$$  This is still a primitive $2^{s-1}$st root of unity because $s>2$.
\end{proof}

\begin{proposition}\label{pr:cyclopair}
    Let $m \in \N$, and suppose that $K/F$ is cyclotomic; assume further that $\omega>1$ if $p=2$.  Let $d \in \Z$ be equivalent to the cyclotomic character modulo $p^{\min\{m,\nu\}}$.  Then $((-\infty,\dots,-\infty),d)$ is a minimal norm pair of length $m$.
\end{proposition}

\begin{proof}
First, observe that since we are not in the case $p=2$ and $\omega = 1 < \nu$, we must have $\omega+n = \nu<\infty$.  

    Let $\xi_{p^\nu}$ be a primitive $p^\nu$th root of unity, and set $\alpha=\xi_{p^{\nu}}$. Then there exists a cyclotomic character $\chi:G\to \Z/p^\nu\Z$ such that $g\xi_{p^\nu}=\chi(g)\xi_{p^\nu}$ for all $g\in G$.  Let $d\in \Z$ be any lift of $\chi(\sigma)$.  Then
    \begin{equation}\label{eq:rootof1}
        \sigma\alpha=d\alpha.
    \end{equation}

    Taking \eqref{eq:rootof1} to the power $p^{n}$ and recalling that $F$ contains all $p^\omega$th roots of unity by definition, it follows that $$p^n\xi_{p^\nu} = \sigma\left(p^n\xi_{p^{\nu}}\right)  = p^n\left(\sigma\xi_{p^{\nu}}\right) = dp^n\xi_{p^\nu}.$$ Hence $d\in U_{\nu-n} = U_\omega$.
    
    Now suppose that $d\in U_{\omega+1}$. Taking \eqref{eq:rootof1} to the power $p^{n-1}$, we see that $p^{n-1}\left(\xi_{p^\nu}\right)$ is a primitive $p^{\omega+1}$st root of unity which is fixed by $\sigma$. But then $F$ contains a primitive $p^{\omega+1}$st root of unity, contrary to the definition of $\omega$.  Hence $d\in U_{\omega}\setminus U_{\omega+1}$ and we may write $d=1+p^{\omega}x$ for $x\in \Z\setminus p\Z$.

    By Lemma \ref{le:norm.of.root.of.unity} we have $N_{K/F}(\xi_{p^{\nu}})=\xi_{p^{\omega}}$ for some primitive $p^{\omega}$th root of unity $\xi_{p^{\omega}}$.
	Then
    \begin{equation*}
        {\frac{d-1}{p}}N_{K/F}\alpha =
        {xp^{\omega-1}}\xi_{p^{\omega}} = \xi_p
    \end{equation*}
    for some primitive $p$th root of unity $\xi_p$. Hence $(\alpha,\{1,1,\dots, 1\})$ represents the norm pair $((-\infty,\dots,-\infty),d)$.  If $m\le \omega$ then $((-\infty,\dots,-\infty),d)$ is the smallest possible norm pair, and we are done.

    Now suppose that $m>\omega$ and there exists a smaller norm pair $((-\infty,\dots,-\infty),\check d)$ represented by $(\check\alpha, \{\check \delta_i\})$, where $\check \delta_i = 0$ for $0\le i<m$ since $a_i = -\infty$. 
    Then $\check d\in U_{\omega+1}$ and by definition of norm pair, for some $x\in \Z$ with $\check d = 1+p^{\omega+1}x$ we have \begin{equation*}
        {xp^{\omega}}(N_{K/F}\check\alpha) + {p^{m-1}}N_{K/F}(\check\delta_m)= \xi_p.
    \end{equation*}
    But the left-hand side is a $p^{\omega}$th power in $F^\times$, so that $F$ contains a primitive $p^{\omega+1}$st root of unity; this contradicts the definition of $\omega$.  Hence $((-\infty,\dots,-\infty), d)$ is a minimal norm pair.
\end{proof}

\subsection{Inequalities on $\mathbf{a}$ involving $\nu$}

Now we treat those cases when $K/F$ is not cyclotomic and for which $\nu<\infty$.  Our method will be to prove that the elements of $\mathbf{a}$ satisfy certain inequalities that involve $\omega$ and $\nu$, and then to use these inequalities to establish that $d$ is the cyclotomic character.

The following definition will be useful for phrasing some of our results.
\begin{definition*}
    Given $m\in \N$ and $\mathbf{a}\in \{-\infty,0,\dots,n\}^{m}$, for
    each $0\le i\le m$, the \emph{index $i^*$ of the last
    nonnegative entry} of $\mathbf{a}$ is defined to be
    \begin{equation*}
        i^* =
        \begin{cases}
            -\infty, & \text{if }a_j=-\infty \text{ for all } 0\le j\le i \\
            \max\{0\le j\le i:a_j\neq -\infty\}, & \text{otherwise.}
        \end{cases}
    \end{equation*}
\end{definition*}

The main result we establish in this section is the following

\begin{proposition}\label{pr:aiversusnu}
    Assume the hypotheses of Theorem \ref{th:d.is.cyclotomic.character}. Then for all $0 \leq i <m$ we have
    \begin{equation*}
        a_i\le n+i-\nu.
    \end{equation*}
    If $m>\nu$ and $K/F$ is not cyclotomic, equality is achieved if and only if $i=\nu^*$.
\end{proposition}



We begin by proving part of Proposition~\ref{pr:aiversusnu} without so much machinery.

\begin{lemma}\label{le:dinuomega}
Suppose $m\in \N$ and $(\mathbf{a},d)$ is a minimal norm pair of length $m$.  If $\omega<\infty$ then $d\in U_{\min(\omega,m)}$, and for each $0\le i<\min(\omega,m)$ we have $$a_i\le n+i-\omega.$$ If $\nu=\omega<m$ then equality is achieved if and only if $i = \nu^*$.  If $\omega < \nu \leq m$, then for each $0\le i<\omega+1$ we have $a_i\le n+i-(\omega+1)$, but $d\not\in U_{\omega+1}$.  
    \end{lemma}

\begin{proof}
    We begin by assuming that $m\ge \omega$, for if $m<\omega$ then we claim that we can recover the desired result by an application of Lemma \ref{le:contracting.and.d.equivalence}.  Indeed, if $m<\omega$ and we assume that $((\hat a_0,\cdots,\hat a_{\omega-1}),\hat d)$ is a minimal norm pair of length $\omega$ that satisfies the desired result, then observe that Lemma \ref{le:contracting.and.d.equivalence} tells us that $((\hat a_0,\cdots,\hat a_{m-1}),\hat d)$ is a minimal norm pair of length $m$ as well.  Applying the definition of minimality, this means that $\hat a_i = a_i$ for all $0 \leq i \leq m-1$, and furthermore that $\min\{v_p(\hat d-1),m\} = \min\{v_p(d-1),m\}$.  Since we assume the desired result for $\hat a_i$, we have $a_i = \hat a_i \leq n+i-\omega$; likewise since we assume $\hat d \in U_\omega$, we get $d \in U_m$.
    
   Hence we assume that $m \geq \omega$.  We prove the result by induction on $s$, showing that for all $1 \leq s \leq \omega$  we have $d\in U_s$, and for any $0 \leq i < m$ we have $a_i\le n+i-s$. For the base case $s=1$, note first that $d \in U_1$ by the definition of the norm pair.  In a similar way, we have $a_0 \leq n-1$ by definition, and also that each $a_i \leq n$ for all $1 \leq i < m$.  Hence the base case is settled.  If $\omega=1$ we are done; otherwise assume that for some $s$ satisfying $1\le s<\omega$ we have $d\in U_s$ and $a_i\le n+i-s$ for all $0 \leq i < m$.  We will now show that $d \in U_{s+1}$, and that $a_i \leq n+i-(s+1)$ for all $0 \leq i < m$.

    By the definition of norm pair,
    \begin{align}\label{eq:norm.pair.equation.in.proof.of.dinuomega}
        \xi_p &= \frac{d-1}{p} N_{K/F}(\alpha)+ N_{K_{n-1}/F}(\delta_0)+\sum_{i=1}^m p^{i-1} N_{K/F}(\delta_i)
        \\&= \frac{d-1}{p} N_{K/F}(\alpha)+ \sum_{i=0}^{m} p^{n-1+i-a_i} N_{K_{a_i}/F}(\delta_i),
    \end{align}
     where here we have defined $a_m =n$.  We claim that the right hand side is a $p^{s-1}$st power. Note first that $n-1+m-a_m= m-1\ge \omega-1>s-1$.  In a similar way, for each $0 \leq i < m$ we can rewrite our inductive assumption to recover $n-1+i-a_i\ge s-1$.  Finally, since $d\in U_s$ we have that $p^{s-1}$ divides $\frac{d-1}{p}$.  As a result we may take $p^{s-1}$th roots of both sides. For some primitive $p^s$-th root of unity $\xi_{p^s}$, then, we have
    \begin{equation}\label{eq:higher.root.norm.pair.equation}
        \xi_{p^s} = \frac{d-1}{p^s} N_{K/F}(\alpha)+ \sum_{i=0}^m p^{n-s+i-a_i}N_{K_{a_i}/F}(\delta_i).
    \end{equation}

    We consider (\ref{eq:higher.root.norm.pair.equation}) modulo $K^{\times p}$. The left-hand side is trivial since $s<\omega$.  Because $n-s+m-a_m>0$, we may ignore the $N_{K_{a_m}/F}(\delta_m)$ factor.  We already know that for $0\le i<m$ we have $n-s+i-a_i\ge 0$, so suppose that there exists $i$ with $0\le i<m$ such that $n-s+i-a_i=0$.  Now Proposition \ref{pr:minimality.conditions}(\ref{it:exc.mod.indecom.condition...ai.vs.aj.inequalityIII}) tells us that $\{j-a_j: a_j \neq -\infty\}$ is a strictly decreasing set, so if $n-s+i-a_i=0$ then for all $j<i$ we have $n-s+j-a_j>0$, whereas for all $i<j$ we have $a_j=-\infty$.  We therefore have
    \begin{equation*}
        \frac{d-1}{p^s} N_{K/F}(\alpha)+ N_{K/F}(\delta_{i}) \in
        pK^{\times}.
    \end{equation*}
    By \cite[Prop.~4.7]{MSS2b}, $N_{K/F}(\alpha)\in pK^{\times}$ and so $N_{K/F}(\delta_i)\in pK^{\times}$.  We obtain that  $\langle [\delta_{i}]_1\rangle$ is not a free $\Fp G_{a_{i}}$-module in $J_1$, contradicting Proposition~\ref{pr:minimality.conditions}(\ref{it:exc.module.indecom.condition...deltafree}). Hence for all $i$ with $0\le i<m$, we have $n-s+i-a_i>0$, or $a_i \le n-(s+1)+i$.

    Considering (\ref{eq:higher.root.norm.pair.equation}) modulo $pF^{\times}$, we obtain
    \begin{equation*}
        \frac{d-1}{p^s}N_{K/F}(\alpha)\in  pF^{\times}.
    \end{equation*}
    If $p$ does not divide $\frac{d-1}{p^s}$, then $\frac{1}{p}N_{K/F}(\alpha)\in F^\times$ and is annihilated by $\sigma-1$, contradicting the second statement of \cite[Prop.~4.7]{MSS2b}.  Hence $d\in U_{s+1}$, as desired, and our
    induction is complete.

    Now assume that $\nu=\omega$.  Since $N_{K/F}(\alpha)\in K^{\times p}$ and $d\in U_\omega$ we can in fact take a $p^\omega$-th root of $\frac{d-1}{p}N_{K/F}(\alpha)$ in $K^\times$.  If $m>\omega$ and for each $i$ with $0\le i<m$ we have $a_i<n+i-\omega$, then we would be able to
    take a $p^\omega$-th root of $\xi_p$ in $K^\times$ in (\ref{eq:norm.pair.equation.in.proof.of.dinuomega}), which is a contradiction.  Hence $a_i = n+i-\omega$ for some $0 \leq i <m$.  Notice that because $\{j-a_j: a_j \neq -\infty\}$ is strictly decreasing, we can have equality $a_i = n+i-\omega$ only at $i = (m-1)^*$; since equality for $i>\omega=\nu$ is impossible, we therefore have equality precisely at $i = \nu^*$.

    If $\nu>\omega$, then apply the same inductive argument with $m\ge \omega+1$ and reaching $s=\omega$ to obtain that for all $i$ with $0\le i<m$, we have $a_i\le n+(\omega+1)-i$.  The  equation
    \begin{equation*}
        \xi_{p^\omega} = \frac{d-1}{p^\omega} N_{K/F}(\alpha) + \sum_{i=0}^m p^{n-\omega+i-a_i} N_{K_{a_i}/F}(\delta_i).
    \end{equation*}
    modulo $pF^{\times}$ is then
    \begin{equation*}
        \xi_{p^\omega} = \frac{d-1}{p^\omega} N_{K/F}(\alpha) \quad
        \mod pF^{\times}.
    \end{equation*}
    By definition of $\omega$, the left-hand side is not trivial modulo $pF^{\times}$, and as above, by the second statement of \cite[Prop.~4.7]{MSS2b}, we know that $N_{K/F}(\alpha)\not\in pF^{\times}$.  Hence it cannot be the case that $p$ divides $\frac{d-1}{p^\omega}$, and we conclude that $d\not\in U_{\omega+1}$.
\end{proof}

\begin{proof}[Proof of Proposition~\ref{pr:aiversusnu}]
    The case $\nu=\infty$ is a consequence of Proposition~\ref{pr:nuinfty}, and the case $\omega = \nu$ was handled in Lemma \ref{le:dinuomega}.  The cyclotomic case was handled in Proposition \ref{pr:cyclopair}. Assume therefore that $\nu<\infty$; since we assume that we are not in the case $p = 2$ and $\omega = 1<\nu$, we may also assume $\omega<\nu <\omega+n$ and that $\omega \geq 2$ when $p=2$.  
    
    By Lemma \ref{le:contracting.and.d.equivalence} we may assume without loss that $m$ is greater than $\nu$.  We claim first that for all $1\le s\le \nu$ and $0\le i<m$, we have $a_i\le n+i-s$; we proceed by induction on $s$.

    The base case $s=1$ follows directly from the definition of norm pair.  Assume then that for some $1\le s<\nu$ we have $a_i\le n+i-s$ for all $0\le i<m$.

    Recalling that $N_{K_j/F}$ and $P(j,0)$ act identically, and $N_{K/F}$ acts as $p^{n-j}P(j,0)$ on $K_j^\times$, we have by definition of norm pair,
    \begin{equation*}
        \xi_p = \frac{d-1}{p}P(n,0)\alpha + \sum_{i=0}^m
        p^{n+i-1-a_i}P(a_i,0)\delta_i.
    \end{equation*}

    Observe that $P(n,0)=\phi_d(P(n,0)) + (\sigma-d)T_d(n)$.
    When $d\neq 1$, $\phi_d(P(n,0))=\frac{d^{p^n}-1}{d-1}$.  Since by definition of norm pair, $(\sigma-d)\alpha=
    \sum_{i=0}^{m} p^i\delta_i$, we obtain for all $d$
    \begin{equation*}
        \xi_p = \frac{d^{p^n}-1}{p}\alpha + \sum_{i=0}^m \left(\frac{d-1}{p}p^iT_d(n) + p^{n+i-1-a_i}P(a_i,0)\right)\delta_i.
    \end{equation*}

    Observe that we cannot be in the case $d = -1$, since for $p>2$ we have $-1 \not\in U_1$, and if $p=2$ then we have assumed we are in the case where $\omega \geq 2$, and by Lemma \ref{le:dinuomega} we have $d \in U_\omega \subseteq U_2$.  We can therefore apply Lemma~\ref{le:tdibar} to the previous equation with $a_i$ in place of $j$ to deduce 
    \begin{equation}\label{eq:aiversusnu}
        \xi_p = 
        	\begin{cases}
	        	\displaystyle \sum_{i=0}^m \left(p^{n+i-1-a_i}P(a_i,0)\right)\delta_i, &d=1\\
        	\frac{d^{p^n}-1}{p}\alpha +\displaystyle \sum_{i=0}^m \left( a T_d(a_i) +  b Q_d(a_i,0) + c P(a_i,0)\right)\delta_i,&d \neq 1
			\end{cases}
    \end{equation}
    where $$a = \frac{d-1}{p} p^{n+i-a_i}, \quad b = \frac{d-1}{p}p^i \left(\sum_{k=0}^{p^{n-a_i}-1} \frac{d^{kp^{a_i}}-1}{d-1}\right), \quad c = p^{n+i-1-a_i}.$$

    We examine the terms of equation (\ref{eq:aiversusnu}) when $d \neq 1$, showing that each is divisible by $p^s$.  First, consider the coefficient of $\alpha$.  Since we are either in the case $p>2$ or $p=2$ and $d \in U_\omega \subseteq U_2$, Lemmas~\ref{le:upowerIII} and \ref{le:dinuomega} give $d^{p^n}\in U_{\omega+n}$, and since $\nu< \omega + n$ we therefore have $d^{p^n} \subseteq U_{\nu+1}$. Since $s<\nu$, we therefore have $p^s$ divides $\frac{d^{p^n}-1}{p}$. (Obviously we have the stronger statement that $p^\nu$ divides $\frac{d^{p^n}-1}{p}$; we will use this fact later in the proof.)

	Now consider the coefficient of $T_d(a_i)$. Since $d \in U_\omega$ by Lemma \ref{le:dinuomega}, this coefficient is divisible by $p^{\omega-1+n+i-a_i}$.  Since by assumption $a_i\le n+i-s$, the coefficient is divisible by $p^{\omega-1+s}$ and hence divisible by $p^s$.

    Now we turn to the coefficient of $Q_d(a_i)$ when $p>2$. We have $$\left(\sum_{k=0}^{p^{n-a_i}-1} \frac{d^{kp^{a_i}}-1}{d-1}\right)=\frac{\frac{d^{p^n}-1}{d^{p^{a_i}}-1}-p^{n-a_i}}{d-1}.$$
    Now write $d=1+p^\omega u$. Then by Lemma~\ref{le:upowerpnot2III}, $d^{p^{a_i}}=1+p^{\omega+a_i}uv$ for some $v\in U_\omega$ 
    and $d^{p^n}=1+p^{\omega+n}uvw$ for some $w\in U_{\omega+a_i}$ of the form $w=1+fuvp^{\omega+a_i}+guvp^{\omega+a_i+1}$ with $f,g\in \Z$
    .  Our fraction is
    \begin{align*}
        \frac{\frac{p^{\omega+n}uvw}{p^{\omega+a_i}uv}-p^{n-a_i}}{p^\omega u} &= \frac{p^{n-a_i}w-p^{n-a_i}}{p^{\omega}u} = p^{n-\omega-a_i}\frac{w-1}{u}\\
        &= p^{n}\frac{fuv+puvg}{u}=
        p^n(fv+pvg).
    \end{align*}
    We conclude that the coefficient of $Q_d(a_i)$ is divisible by $p^{\omega+n-1+i}$.  Since $s<\nu< \omega+n$, we have $s\le \nu-1< \omega+n-1+i$, and so the coefficient is divisible by $p^{s}$.

    If instead $p=2$, we consider the coefficient of $Q_d(a_i)$ as follows.  As before, we have $$\left(\sum_{k=0}^{2^{n-a_i}-1} \frac{d^{k2^{a_i}}-1}{d-1}\right)=
        \frac{\frac{2^{p^n}-1}{2^{p^{a_i}}-1}-2^{n-a_i}}{d-1}.$$    Now write $d=1+2^\omega u$.  Then by Lemma~\ref{le:upowerpnot2III}, $d^{2^{a_i}}=1+2^{\omega+a_i}uv$ for $v\in U_{\omega-1}$ and
    $d^{2^n}=1+2^{\omega+n}uvw$ for $w\in U_{\omega+a_i-1}$
    of the form $1+2^{\omega+a_i-1}uv+2^{\omega+a_i}uve$ for $e\in \Z$.  Our fraction is
    \begin{align*}
        \frac{\frac{2^{\omega+n}uvw}{2^{\omega+a_i}uv}-2^{n-a_i}}{2^\omega u} &= \frac{2^{n-a_i}w-2^{n-a_i}}{2^{\omega}u} = 2^{n-\omega-a_i}\frac{w-1}{u}\\
        &= 2^{n-1}\frac{uv+2uve}{u}=
        2^{n-1}(v+2ve).
    \end{align*}
    We conclude that the coefficient of $Q_d(a_i)$ is divisible by $2^{\omega+n-2+i}$.  Since  $s<\nu< \omega+n$, we have $s\le \nu-1\le \omega+n-2+i$, and so the coefficient is divisible by $2^{s}$.

	Now consider equation \eqref{eq:aiversusnu} modulo $p^sK^\times$.  By the foregoing discussion, the right hand side is
    \begin{equation*}
        \sum_{i=0}^m p^{n+i-1-a_i}P(a_i,0)\delta_i \quad \mod p^sK^\times.
    \end{equation*}
    By assumption, for each $i<m$ we have $a_i\le n+i-s$, and so $n+i-1-a_i\ge s-1$.  Moreover, when $a_m \neq -\infty$ we have $n+m-1-a_m\ge m-1\ge s$, and so we may ignore this term.

    Suppose that there exists $i$ with $0\le i<m$ satisfying $n+i-1-a_i=s-1$.  By Proposition \ref{pr:minimality.conditions}(\ref{it:exc.mod.indecom.condition...ai.vs.aj.inequalityIII}) the set $\{j-a_j:a_j \neq -\infty\}$ is strictly decreasing, so if $n+i-s-a_i=0$ then for all $j<i$ we have $n+j-s-a_j>0$, whereas for all $i<j$ we have $a_j = -\infty$.  Taking $p^{s-1}$st roots of unity on both sides, we therefore obtain that 
    \begin{equation*}
        \xi_{p^s} =
        P(a_i,0)\delta_i + pk
    \end{equation*}
    for some primitive $p^s$th root of unity $\xi_{p^s}$ and some $k\in K^\times$.  Since $s<\nu$, the left-hand side lies in $pK^\times$.  Since $P(a_i,0) \neq 0 \in \F_pG_{a_i}$, we conclude that $\langle [\delta_i]_1 \rangle_{\F_pG_{a_i}}$ is not free, contradicting Proposition~\ref{pr:minimality.conditions}(\ref{it:exc.module.indecom.condition...deltafree}).

    For the last statement of the proposition, suppose to the contrary that
    $$a_i<n+i-\nu$$ for each $0\le i<m$.  Notice that if $a_m \neq -\infty$ then we also have $a_m \leq n < n+m-\nu$, so this inequality also includes $i=m$.  Momentarily we will prove that we can extract a $p^\nu$th root of the right hand side of equation (\ref{eq:aiversusnu}), and hence $K^\times$ contains a primitive $p^{\nu+1}$st root of unity, a contradiction. We will therefore have $a_i = n+i-\nu$ for some $0 \leq i < m$.  Since $\{j-a_j: a_j \neq -\infty\}$ is strictly decreasing, equality can happen only for $i = (m-1)^*$.  Since equality is impossible for $i>\nu$, this implies equality occurs precisely at $i = \nu^*$.

Now we justify our assertion that the right side of equation (\ref{eq:aiversusnu}) is a $p^\nu$th power.  First, recall that we have already shown that the coefficient of $\alpha$ in (\ref{eq:aiversusnu}) is divisible by $p^\nu$.  We also showed that the coefficient of $T_d(a_i)$ is divisible by $p^{\omega+n-1+i-a_i}$, and since we're assuming that $a_i<n+i-\nu$ we see that this coefficient is divisible by $p^\nu$.  Obviously the coefficient of $P(a_i,0)$ is divisible by $p^\nu$ by our assumption.  Finally, when $p>2$ we have shown that the coefficient of $Q_d(a_i)$ is divisible by $p^{\omega +n-1+i}$, and hence by $p^\nu$ since we have $\nu < \omega +n$.  When $p=2$, on the other hand, we have shown that $Q_d(a_i)$ is divisible by $p^{\omega+n-2+i}$.  Hence the coefficient of $Q_d(a_i)$ is divisible by $p^\nu$ if $\nu < \omega +n-1$, or if $\nu = \omega + n-1$ and $i>0$.  

We argue now that if $\nu = \omega+n-1$, then in fact $a_0 = -\infty$, so that we still reach the desired result.  To see this, one uses Lemma \ref{le:norm.of.root.of.unity} to argue that $N_{K/F}\left(\xi_{2^\nu}\right) = \xi_{2^{\omega-1}}$.  In particular we have $-1 = N_{K/F}(\beta)$ for some $\beta \in K^\times$; squaring this and applying Hilbert's Theorem 90 gives $\alpha \in K^\times$ such that $\alpha^{\sigma-1} = \beta^2$.  By setting $\delta_0 = 1$ and $\delta_1 = \beta$, we see that $(\alpha,\{\delta_0,\delta_1\})$ represents the norm pair $((-\infty),1)$ of length $1$.  Since $((a_0),d)$ is a minimal norm pair of length $1$ by Lemma \ref{le:contracting.and.d.equivalence}, we therefore have $a_0 = -\infty$ as desired.
	\end{proof}

\subsection{Proof of Theorem \ref{th:d.is.cyclotomic.character}}

	We begin by proving (\ref{it:d.is.cyclo.character...any.d.works}).  So suppose $\omega<\nu$, and let $\tilde d \equiv d \pmod{p^\nu}$. If $m\le \nu$ then $(\mathbf{a},\tilde d)$ is a norm pair by Lemma \ref{le:contracting.and.d.equivalence}; it must be minimal since $$\min\{v_p(d-1),m\} = \min\{\omega,m\} = \min\{v_p(\tilde d-1),m\}.$$ Hence we assume that $\nu<m$.

    Suppose first that $K/F$ is cyclotomic.  Then as in the proof of
    Proposition \ref{pr:cyclopair}, $(\xi_{p^{\nu}},\{1,1,\dots,1\})$
    represents the minimal norm pair $((-\infty,\dots,-\infty),d)$, where $d$
    satisfies $(\sigma-d)\xi_{p^\nu}=0$.  
    Then
    $$(\sigma-\tilde d)\xi_{p^{\nu}}=(\sigma-d)\xi_{p^\nu} + (d-\tilde d)  \xi_{p^\nu} = 0,$$ and so
    $(\xi_{p^\nu},\{1,1,\dots,1\})$ represents $((-\infty,\dots,-\infty),\tilde d)$ as well. Moreover,
    $v_p(d-1)=\omega=v_p(\tilde d-1)$ since $\omega<\nu$.  Hence $((-\infty,\dots,-\infty),\tilde d)$ is a minimal norm pair.  We assume for the
    remainder of the proof that $K/F$ is not cyclotomic. Since
    $\nu<\infty$ and we are not in the case $p=2$ and $\omega = 1<\nu$, we have $\nu<\omega+n$.
    
    We prove by induction on $\nu\le s\le m$ that if $(\mathbf{a},d) = ((a_0,\dots,a_{m-1}),d)$ is a minimal norm pair of length $m$ and $\tilde
    d=d\mod p^\nu$, then $((a_0,\dots,a_{s-1}),d_s)$ is a minimal norm pair
    of length $s$ for any $d_s=\tilde{d} \mod p^s$.  Recall that Lemma \ref{le:contracting.and.d.equivalence} gives that $((a_0,\dots,a_{s-1}),d)$ is a minimal norm pair of length $s$.

    Since $\tilde d=d\mod p^\nu$, Lemma \ref{le:contracting.and.d.equivalence} proves the case $s=\nu$. Now suppose that $s>\nu$ and that $((a_0,\dots,a_{s-2}),d_{s-1})$ is a minimal norm pair of length $s-1$ with $\tilde d=d_{s-1} \mod
    p^{s-1}$.  Write $d_s= d_{s-1}+p^{s-1}x$ for $x \in \Z$.

    Let $(\alpha,\{\delta_i\}_{i=0}^{s-1})$ represent $((a_0,\dots,a_{s-2}),d_{s-1})$.  Then
    by definition
    \begin{align*}
      (\sigma-d_{s-1})\alpha &= \sum_{i=0}^{s-1} p^i\delta_i\\
        \xi_p &= \frac{d_{s-1}-1}{p}N_{K/F}(\alpha) +
        N_{K_{n-1}/F}(\delta_0) + \sum_{i=1}^{s-1}
        p^{i-1}N_{K/F}(\delta_i).
    \end{align*}
    For $0\le i<s-1$, let $\check\delta_i=\delta_i$;
    set $\check\delta_{s-1}=\delta_{s-1}-x \alpha$ and $\check \delta_{s}=1$.  Then
    \begin{align*}
        (\sigma-d_{s})\alpha &= (\sigma-d_{s-1})\alpha + (d_{s-1}-d_{s})
        \alpha = \sum_{i=0}^{s-1} p^i\delta_i - p^{s-1} x \alpha
		= \sum_{i=0}^{s} p^i\check\delta_i
    \end{align*}
    and
    \begin{align*}
        \frac{d_s-1}{p}&N_{K/F}(\alpha) + N_{K_{n-1}/F}(\check\delta_0) + \sum_{i=1}^{s} p^{i-1}N_{K/F}(\check\delta_i)  
        \\=&~ \frac{d_s-d_{s-1}}{p}N_{K/F}(\alpha) + \frac{d_{s-1}-1}{p}N_{K/F}(\alpha)
        \\ &\quad + N_{K_{n-1}/F}(\delta_0) + \sum_{i=1}^{s} p^{i-1}N_{K/F}(\delta_i) - p^{s-2}N_{K/F}(x \alpha)
        \\ =&~p^{s-2}xN_{K/F}(\alpha) - p^{s-2}N_{K/F}(x\alpha) 
        \\ &\quad + \frac{d_{s-1}-1}{p}N_{K/F}(\alpha) + N_{K_{n-1}/F}(\delta_0) + \sum_{i=1}^{s-1} p^{i-1}N_{K/F}(\delta_i) \\=& ~\xi_p.
    \end{align*}
    Let $\check{\mathbf{a}}=(\check a_0,\dots,\check a_{s-1})$ where $\check a_i=a_i$ for all $0\le i<s-1$, and $\check a_{s-1} = n$.  Then $(\alpha, \{\check \delta_i\})$ represents the norm pair $((\check a_0,\cdots,\check a_{s-1}),d_{s})$ of length $s$. 
    
    Since $K/F$ is not cyclotomic and $s>\nu$, by Proposition~\ref{pr:aiversusnu} we have $a_{\nu^*}+\nu-\nu^*= n$.  Suppose first that we are in the case $s = \nu+1$; then either we have $\nu^* = \nu$ or $\nu^* = (\nu-1)^*$.  In the first case we have $a_\nu = n$, and since we have already observed that $((a_0,\dots,a_\nu),d)$ is a minimal norm pair of length $\nu+1$, we must have
    $$((a_0,\dots,a_\nu),d) \leq ((\check a_0,\dots,\check a_\nu),d_{\nu+1}).$$  Hence we have $(a_0,\dots,a_{\nu}) = (\check a_0,\dots,\check a_{\nu})$, and furthermore since $\omega<\nu<m$, we also have
    $$\min\{v_p(d_s-1),m\} = \omega = \min\{v_p(d-1),m\}.$$
    So $((a_0,\dots,a_{\nu}),d_{\nu+1})$ is a minimal norm pair of length $s=\nu+1$, as desired.
    
    Alternatively, if we are in the case $s = \nu+1$ and $\nu^* = (\nu-1)^*$, then it must be that $a_\nu = -\infty$.  We also have
    $$\check a_{\nu} = n = a_{(\nu-1)^*}+\nu-(\nu-1)^* = \check a_{(\nu-1)^*}+\nu-(\nu-1)^*.$$
    By Lemma \ref{le:delicateIII},
    $$((\check a_0,\dots,\check a_{\nu-1},-\infty),d_{\nu+1}) = ((a_0,\dots,a_{\nu-1},a_\nu),d_{\nu+1})$$
    is a norm pair of length $\nu+1=s$. (Note that we can apply Lemma \ref{le:delicateIII} in the case $p=2$ since by assumption $\omega<\nu$, and hence $\omega \geq 2$ by our overriding assumption; then Lemma \ref{le:dinuomega} gives $d \in U_\omega \subseteq U_2$.)   Additionally, since $\omega < \nu < m$ we again have
    $$\min\{v_p(d_s-1),m\} = \omega = \min\{v_p(d-1),m\}$$
    and so $((a_0,\dots,a_\nu),d_{\nu+1})$ is a minimal norm pair of length $s = \nu+1$, as desired.
    
    Now suppose we are in the case $s>\nu+1$.  Observe that then
    \begin{equation*}
        \check a_{s-1} \leq  n = a_{\nu^*}+\nu-\nu^* < a_{\nu^*}+ s-1 -\nu^*= \check a_{\nu^*}+s-1-\nu^*.
    \end{equation*}
    Hence by Lemma \ref{le:delicateIII},
    \begin{equation*}
       ((\check a_0,\dots,\check a_{s-2},-\infty),d_s) = ((a_0,\dots,a_{s-2},-\infty),d_s)
    \end{equation*}
is a norm pair of length $s$. The same argument we carried out above shows that $((a_0,\dots,a_{s-1}),d_s)$ is a minimal norm pair.
       
       Now we settle (\ref{it:d.is.cyclo.character...omega.equals.nu}).  If $\omega = \infty$ then the result follows from Proposition \ref{pr:nuinfty}. If $m \leq \omega < \infty$ then Lemma \ref{le:dinuomega} gives $d \in U_{\min\{\omega,m\}} = U_m$ as desired.  So assume that $m > \omega = \nu$.  We will prove by induction on $\nu \leq s \leq m$ that if $(\mathbf{a},d)=((a_0,\cdots,a_{m-1}),d)$ is a minimal norm pair of length $m$, then $((a_0,\dots,a_{s-1}),1)$ is a norm pair of length $s$.  Since Lemma \ref{le:contracting.and.d.equivalence} gives that $((a_0,\dots,a_{s-1}),d)$ is a minimal norm pair of length $s$, this will prove that $d \in U_s$, and in particular $d \in U_m$. 
       
       The proof proceeds precisely as it did in the the proof of (\ref{it:d.is.cyclo.character...any.d.works}), so we only provide a sketch.  The base case $s = \nu$ follows from Lemma \ref{le:contracting.and.d.equivalence} because $d \in U_\omega = U_\nu$.  Otherwise we assume $\omega < s < m$ and that $((a_0,\dots,a_{s-2}),1)$ is a norm pair of length $s-1$ represented by $(\alpha,\{\delta_i\}_{i=0}^{s-1})$.  We define $\check \delta_i = \delta_i$ for all $0 \leq i \leq s-1$ and set $\check \delta_s = 1$.  One then verifies that $(\alpha,\{\check \delta_i\}_{i=0}^{s}\}$ represents $((a_0,\dots,a_{s-2},n),1)$, and shows (through case analysis, first when $s = \nu+1$ and then when $s>\nu+1$) that $((a_0,\dots,a_{s-2},a_{s-1}),1)$ is a norm pair. 
Finally, we establish (\ref{it:d.is.cyclo.character...d.is.cyclo.character}).  We have already settled the case $\nu = \infty$ and the cyclotomic case in Propositions \ref{pr:nuinfty} and \ref{pr:cyclopair}, so assume that $K/F$ is not cyclotomic and $\nu<\infty$.  Since we assume we are not in the case $p=2$ and $\omega = 1<\nu$, this allows us to conclude that $\nu < \omega + n$. By Lemma \ref{le:contracting.and.d.equivalence} we may assume that $m$ is
    larger than $\nu$.  
    
    We would like to use equation (\ref{eq:aiversusnu}) from the proof of Proposition \ref{pr:aiversusnu}, but to do so we must first rule out the possibility that $d = -1$.  This is immediate for $p>2$ since $-1 \not\in U_1$.  If $p=2$, our assumption gives us that either $\omega = \nu = 1$ or $\omega \geq 2$.  In the former case, we just proved that $d \in U_m$, and since we assume $m>\nu=1$ we have $d \in U_2$; in the latter case Lemma \ref{le:dinuomega} gives $d \in U_2$.  Hence $d \neq -1$, and so 
   \begin{equation*}
        \xi_p = 
        	\begin{cases}
	        	\displaystyle \sum_{i=0}^m \left(p^{n+i-1-a_i}P(a_i,0)\right)\delta_i, &d=1\\
        	\frac{d^{p^n}-1}{p}\alpha +\displaystyle \sum_{i=0}^m \left( a T_d(a_i) +  b Q_d(a_i,0) + c P(a_i,0)\right)\delta_i,&d \neq 1
			\end{cases}
    \end{equation*}
    where $$a = \frac{d-1}{p} p^{n+i-a_i}, \quad b = \frac{d-1}{p}p^i \left(\sum_{k=0}^{p^{n-a_i}-1} \frac{d^{kp^{a_i}}-1}{d-1}\right), \quad c = p^{n+i-1-a_i}.$$
    Since $\nu < \omega + n$ and $d \in U_2$ when $p=2$, Lemmas \ref{le:dinuomega} and \ref{le:upowerIII} give $\frac{d^{p^n}-1}{p} \in U_\nu$.  The proof of Proposition \ref{pr:aiversusnu} established that the coefficient of $Q_d(a_i)$ is divisible by $p^{\nu-1}$.  Furthemore, Proposition~\ref{pr:aiversusnu} tells us that for each
    $i$ with $0\le i<m$ we have $a_i\le n+i-\nu$ with equality precisely for $i = \nu^*$.  Hence we may take
    $p^{\nu-1}$ roots of both sides  to get
   \begin{equation*}
        \xi_{p^\nu} = 
        	\begin{cases}
	        	\displaystyle \sum_{i=0}^m \left(p^{n+i-\nu-a_i}P(a_i,0)\right)\delta_i, &d=1\\
        	\frac{d^{p^n}-1}{p^\nu}\alpha +\displaystyle \sum_{i=0}^m \left( \hat a T_d(a_i) +  \hat b Q_d(a_i,0) + \hat c P(a_i,0)\right)\delta_i,&d \neq 1
			\end{cases}
    \end{equation*}
    where $$\hat a = \frac{d-1}{p^\nu} p^{n+i-a_i}, \quad \hat b = \frac{d-1}{p^\nu}p^i \left(\sum_{k=0}^{p^{n-a_i}-1} \frac{d^{kp^{a_i}}-1}{d-1}\right), \quad \hat c = p^{n+i-\nu-a_i}.$$

Applying $\sigma-1$ to both sides when $d=1$, we
    obtain $(\sigma-1)\xi_{p^{\nu}}=1$ since the right-hand
    side is a product of norms $N_{K_{a_i}/F}(\delta_i)$.  Hence $\nu=\omega$ and $d=1$ is equivalent to the cyclotomic
    character modulo $p^\nu$.

    When $d\neq 1
    $, we apply $\sigma-d$ to both sides.
    The right-hand side is
    \begin{multline*}
        \sum_{i=0}^m \left(\frac{d^{p^n}-1}{p^\nu} p^i +(\sigma-d)
        \left( \hat a T_d(a_i) +  \hat b Q_d(a_i,0) + \hat c P(a_i,0)\right)\right)\delta_i.
    \end{multline*}
    Consider the coefficient of $\delta_i$. Taking into account the identities 
    \begin{align*}
    (\sigma-d)T_d(a_i)&=P(a_i,0)-\phi_d(P(a_i,0))\\
    (\sigma-d)Q_d(a_i,0)&=\sigma^{p^{a_i}}-d^{p^{a_i}}=1-d^{p^{a_i}}\\
    (\sigma-d)P(a_i,0)&=(1-d)P(a_i,0),
    \end{align*} we see that this coefficient is
    \begin{align*}
        &\frac{d^{p^n}-1}{p^\nu}p^i + \frac{d-1}{p^\nu} p^{n+i-a_i} P(a_i,0)-\frac{d-1}{p^\nu} p^{n+i-a_i} \frac{d^{p^{a_i}}-1}{d-1} \\
        &\quad + \frac{d-1}{p^\nu}p^i \left(\sum_{k=0}^{p^{n-a_i}-1} \frac{d^{kp^{a_i}}-1}{d-1}\right)(1-d^{p^{a_i}}) + p^{n+i-\nu-a_i}(1-d)P(a_i,0)\\
        &=\frac{d^{p^n}-1}{p^\nu}p^i -  p^{n+i-a_i} \frac{d^{p^{a_i}}-1}{p^\nu} + \frac{d-1}{p^\nu}p^i \frac{\frac{d^{p^n}-1}{d^{p^{a_i}}-1}-p^{n-a_i}}{d-1} (1-d^{p^{a_i}})\\
        &=\frac{d^{p^n}-1}{p^\nu}p^i -  p^{n+i-a_i} \frac{d^{p^{a_i}}-1}{p^\nu} + \frac{-p^i(d^{p^n}-1)-(1-d^{p^{a_i}})p^{n-a_i+i}}{p^\nu}\\
        &=\frac{(d^{p^n}-1)p^i-p^{n+i-a_i}(d^{p^{a_i}}-1)}{p^\nu} +       \frac{-p^i(d^{p^n}-1)-(1-d^{p^{a_i}})p^{n-a_i+i}}{p^\nu}
        \\&=0.
    \end{align*}
    That is, multiplicatively, $\xi_{p^{\nu}}^{\sigma -d}=1$, whence $d$ is equivalent to the cyclotomic character.

\section{Additional Module-theoretic properties of $X$}\label{se:more.module.properties}

In addition to the notation we have already adopted, for $0\le
i\le n$ we also define $H_i$ to be the subgroup of $G$ of degree
$p^{n-i}$, so that $G/H_i=G_i$.  

In this section we will see that there is a connection between the minimal norm vector $\mathbf{a}$ and certain properties of $X$ when considered as an $R_mH_i$-module for the various $0 \leq i \leq n$.  Hence we begin by establish some basic results in this vein.

\subsection{Modules over free quotients of $R_mG$}

Observe that for subgroups $S\le H\le G$, the $R_mS$-module $R_mH$ is a direct sum of $\vert H\vert/\vert S\vert$ modules isomorphic to $R_mS$.  Hence a free $R_mH$-module on $c$ generators is a free $R_mS$-module on $c\vert H\vert/\vert S\vert$ generators.

\begin{definition*}
For a subgroup $H\leq G$, we say that an $R_mH$-module $W$ is an \emph{$H$-eigenmodule} if there exists a homomorphism $e:H\to \left(R_m/\ann_{R_m} W\right)^\times$ such that for all $h \in H$ and $w \in W$ we have
\begin{equation*}
    hw=e(h)w.
\end{equation*}
We call $e$ an \emph{eigenhomomorphism}.  
\end{definition*}

Now suppose $H=\langle \sigma^{p^i}\rangle$ for some $0\le i\le n$, and let $\tau=\sigma^{p^i}$.  Then an $R_mH$-module $W$ is an $H$-eigenmodule if and only if there exists $c\in R_m$ such that $(\tau-c)\subset \ann_{R_mH} W$, and if so then the eigenhomomorphism $e$ is induced by the map $\tau\mapsto c$. Observe that if $S\le H$ and $W$ is an $H$-eigenmodule with eigenhomomorphism $e$, then $W$ is an $S$-eigenmodule with corresponding eigenhomomorphism $e\vert_S$.

\begin{lemma}\label{le:eigenfree}
    Suppose that $m\in \N$, $S\le H\le G$, and $W$ is a cyclic $R_mH$-module which is both an $R_mH$-eigenmodule and a free $R_m(H/S)$-module.  Then $W$ is a trivial $R_mH$-module, that is, the eigenhomomorphism is trivial.
\end{lemma}

\begin{proof}
    Let $w$ be a cyclic generator of $W$, so that $W=R_mHw$.
    On one hand, since $W$ is an $R_mH$-eigenmodule, we have $W=R_m w$.  On the other hand,
    since $W$ is a free cyclic $R_m(H/S)$-module, the minimum
    number of generators of $W$ as an $R_m$-module is $\vert H/S\vert$.  Hence $H=S$.  That is, $W$ is fixed by $H$.  Hence $W$ is a trivial $R_mH$-module.
\end{proof}

\begin{definition*}
    Let $m\geq 2$ and $H\le G$.
    We say that an $R_mH$-module
    $W$ has property $\mathcal{P}(H)$ if there exist $s \geq 0$ and $y,z \in W\setminus (\tau^{p^s}-1,p)W$ satisfying
    \begin{equation*}
        (\tau^{p^s}-1)y=p^{m-1} z.
    \end{equation*}
\end{definition*}
Suppose that $H \leq T \leq G$, and that $W$ is an $R_mT$-module. From the previous definition, it follows that if $W$ satisfies property $\mathcal{P}(H)$ (considered as an $R_mH$-module), then $W$ satisfies property $\mathcal{P}(T)$ (considered as an $R_mT$-module) as well.  

\begin{lemma}\label{le:free.modules.do.not.have.property.P}If $W$ is a cyclic $R_mH$-module which is free as an $R_m(H/S)$-module for some $S\le H$, then $W$ does not have property $\mathcal{P}(H)$.
\end{lemma}

\begin{proof}

We adopt the notation $\langle \tau \rangle = H$ and $\langle \tau^{p^k} \rangle = S$.  Write $w$ for a generator of $W$.  Suppose that for some $s \geq 0$ and $y,z \in \langle w \rangle$ we have $$(\tau^{p^s}-1)y = p^{m-1}z.$$
We will prove that this forces $z \in (\tau^{p^s}-1,p)\langle w \rangle$, from which we can conclude that property $\mathcal{P}(H)$ cannot be satisfied.

First, suppose that $\tau^{p^s} \in S$.  Then we have $(\tau^{p^s}-1)y=0$, and so $p^{m-1}z = 0$.  Lemma \ref{le:kerbasicIII} tells us that $\ann_{R_m(H/S)}\langle p^{m-1} \rangle = \langle p \rangle$, which would then imply $z \in (p)\langle w \rangle$.

Hence assume that $\tau^{p^s} \not\in S$, so that $s<k$.  Define $$\Theta = \sum_{t=0}^{p^{k-s}-1}\tau^{t p^s}.$$  Since $\Theta(\tau^{p^s}-1) = \tau^{p^k}-1 = 0$ in $R_m(H/S)$, we have $$0 = \Theta (\tau^{p^s}-1)y = \Theta p^{m-1} z.$$  Notice that $\Theta p^{m-1} \neq 0 \in R_m(H/S)$, and that $\Theta p^{m-1} p^i (\tau-1)^j = 0$ whenever $i>0$ or $j \geq p^s$.  Hence if we write $$z = \left(\sum_{i=0}^{m-1} \sum_{j=0}^{p^k-1} c_{i,j}p^i  (\tau-1)^j\right) w,$$ then the relation $\Theta p^{m-1} z=0$ implies that $c_{i,j}=0$ when $i=0$ and $0 \leq j <p^s$.  This gives that $z \in ((\tau-1)^{p^s},p)\langle w\rangle$, and since $(\tau-1)^{p^s} \equiv (\tau^{p^s}-1)$ modulo $p$, we have $z \in (\tau^{p^s}-1,p)\langle w \rangle$.
\end{proof}

\begin{corollary}\label{cor:direct.sums.do.not.have.property.P}
Suppose that $W$ is a direct sum of modules, each isomorphic to a free $R_m(H/S)$-module for some $S \leq H$.  Then $W$ does not have property $\mathcal{P}(H)$.
\end{corollary}

\begin{proof}
Write $W = \oplus \langle w_i \rangle$, where each $\langle w_i \rangle$ is free as an $R_m(H/S)$ module for some $S$.  As before, suppose that $s \geq 0$ and that $y,z \in W$ satisfy $$(\tau^{p^s}-1)y = p^{m-1} z.$$  Write $y = \sum y_i$ and $z = \sum z_i$, so that we have equations
$$(\tau^{p^s}-1)y_i = p^{m-1}z_i.$$  By the proof of Lemma \ref{le:free.modules.do.not.have.property.P}, this equation implies that $z_i \in (\tau^{p^s}-1,p)\langle w_i \rangle$, and hence $z \in (\tau^{p^s}-1,p)W$.  Hence Property $\mathcal{P}(H)$ cannot be satisfied.
\end{proof}

\subsection{Results involving interpolated vectors}

\begin{definition*}
    Let $(\mathbf{a},d)$ be a minimal norm pair of length $m$.  Then the \emph{interpolated vector} $\tilde{\mathbf{a}} = (\tilde a_0,\dots,\tilde a_{m-1})$ is defined by $$\tilde a_i = \left\{\begin{array}{ll}-\infty,&\text{ if }i^*=-\infty\\a_{i^*}+(i-i^*),&\text{ otherwise. }\end{array}\right.$$
\end{definition*}
\noindent Observe that if $m\ge 2$, $\mathbf{a} = (a_0,\dots,a_{m-1})$, and $\mathbf{a}' = (a_0,\dots,a_{m-2})$, then
$\tilde a_i=\tilde{a}_i'$ for $0\le i\le m-2$.  Hence the
interpolated vector of a truncation of a vector $\mathbf{a}$ is the same
as corresponding truncation of the interpolated vector $\tilde{\mathbf{a}}$.  Observe also that if $\tilde a_{i-1} \neq -\infty$, then we have
$$\tilde a_i - \tilde a_{i-1} = a_{i^*}-a_{(i-1)^*}+1-(i^*-(i-1)^*) >0.$$

\begin{proposition}\label{pr:interp}
    Suppose that $(\mathbf{a},d)$ is a minimal norm pair of length $m$.  Then $a_0=\tilde a_0$, and for all $1\le i\le m$ we have 
    \begin{equation*}
        a_i = \begin{cases}
            -\infty, & \tilde a_i = \tilde a_{i-1}+1 \\
            \tilde a_i, & \tilde a_i>\tilde a_{i-1}+1.
        \end{cases}
    \end{equation*}
\end{proposition}

\begin{proposition}\label{pr:eigen}
    Assume the hypotheses of Theorem \ref{th:d.is.cyclotomic.character}, let $0 \leq i<\nu$ be given.  Suppose $(\mathbf{a},d)$ is a minimal norm pair of length $i+1$.
    \begin{enumerate}
    \item\label{it:eigen...ai.is.minus.infinity} If $\tilde a_i=-\infty$, then $X$ is an $R_{i+1}G$-eigenmodule. 
    \item\label{it:eigen...not.eigenmodule}If $\tilde a_i\neq -\infty$, then $X$ is not an $R_{i+1}H_{\tilde a_i}$-eigenmodule, but does satisfy property $\mathcal{P}(H_{\tilde a_i})$.  
    \item\label{it:eigen...is.trivial.module}If $\tilde a_i \neq -\infty$, then $X$ is a trivial $R_{i+1}H_{\tilde a_i+ 1}$-module. 
    \end{enumerate}
\end{proposition}

\begin{lemma}\label{le:inter}
Assume the hypotheses of Theorem \ref{th:d.is.cyclotomic.character}.  Suppose $(\mathbf{a},d)$ is a minimal norm pair of length $m$ represented by $(\alpha,\{\delta_i\}_{i=0}^m)$.  If $c\in R_m$ satisfies $c [\alpha]_m \in \langle [\delta_0]_m,\dots,[\delta_{m-1}]_m\rangle$, then $c\in p^\nu R_m$.  (Here we interpret $p^\infty R_m = \{0\}$.)
\end{lemma}

\begin{proof}
Recall that $X = \langle [\alpha]_m,[\delta_0]_m,[\delta_1]_m,\dots,[\delta_{m-1}]_m\rangle$ has its relations generated by the equations from (\ref{eq:definition.of.Xadm}).  
    Suppose 
    \begin{equation*}
        c[\alpha]_m\in \langle [\delta_0]_m,\dots,[\delta_{m-1}]_m\rangle,
    \end{equation*}
    so that there exist $g_0,\dots,g_{m-1}\in R_mG$ such that
    \begin{equation*}
        c[\alpha]_m - \sum_{i=0}^{m-1} g_i[\delta_i]_m= [0]_m.
    \end{equation*}
    Let $\mathcal{X}$ be a free $R_mG$-module on the set of generators $\{A,\Delta_0,\cdots,\Delta_{m-1}\}$, and define a surjection $\mathcal{X} \to X$ by sending $A$ to $[\alpha]_m$, and each $\Delta_i$ to $[\delta_i]_m$.  By (\ref{eq:definition.of.Xadm}), we have
    \begin{equation*}
        cA - \sum_{i=0}^{m-1} g_i\Delta_i = f_\alpha\left((\sigma-d)A
        -\sum_{i=0}^{m-1} p^i\Delta_i\right) +
        \sum_{i=0}^{m-1}f_i(\sigma^{p^{a_i}}-1)\Delta_i
    \end{equation*}
    for $f_\alpha,f_0,f_1,\dots,f_{m-1} \in R_m G$.

    Considering the coefficients of $A$, we obtain
    \begin{equation*}
        c = f_\alpha(\sigma-d).
    \end{equation*}
    Let $s=\min \{m,\nu\}$. Theorem \ref{th:d.is.cyclotomic.character}(\ref{it:d.is.cyclo.character...d.is.cyclo.character}) tells us that $d$ is equivalent to the cyclotomic character modulo $p^s$.  If $\omega \geq m$ then we have $d \in U_m$, in which case clearly $d^{p^n} \in U_s$.  Otherwise $d \in U_\omega$, in which case Lemma \ref{le:upowerIII} tells us $d^{p^n} \in U_{\omega +n}$.  Since we assume we are not in the case $p=2$ and $\omega = 1 < \nu$, it follows that $U_{\omega +n} \subseteq U_\nu \subseteq U_s$.  In either case, then, we have $d^{p^n}\in U_s$, so the map $\phi_d^{(s)}:R_sG \to R_s$ induced by $\sigma \mapsto d$ is well-defined.  The previous equation then becomes
    \begin{equation*}
        c=\phi_d^{(s)}(c)=0,
    \end{equation*}
    which gives the desired result.
\end{proof}

\begin{proof}[Proof of Proposition~\ref{pr:interp}]
    If $\tilde a_0=-\infty$ then $0^*=-\infty$ and therefore
    $a_0=-\infty$.  Otherwise, $0^*=0$ and so $\tilde a_0=a_0$.  

    Now assume that $1\le i\le m$.  If $(i-1)^*=-\infty$ then $\tilde a_i = \tilde a_{i-1}+1$ is equivalent to $\tilde a_i = -\infty$, which is in turn equivalent to $a_i = -\infty$. So assume instead that $(i-1)^*\geq 0$.  
    
    By definition $a_i = -\infty$ if and only if $i^*<i$, which is clearly equivalent to $i^* = (i-1)^*$.  Now recall that Proposition \ref{pr:minimality.conditions}(\ref{it:exc.mod.indecom.condition...ai.vs.aj.inequalityIII}) tells us that the sequence $\{a_j-j:a_j \neq -\infty\}$ is strictly increasing.  Since $a_{i^*},a_{(i-1)^*}\neq -\infty$ by assumption, we have $i^*=(i-1)^*$ is equivalent to 
    $$a_{i^*}-i^* = a_{(i-1)^*}-(i-1)^*.$$  By adding $i$ to both sides, we see that this latter equality occurs if and only $\tilde a_i = \tilde a_{i-1}+1$.  
    
    Since we already know $\tilde a_i > \tilde a_{i-1}$ when $(i-1)^*\geq 0$, the desired result follows.
\end{proof}

\begin{proof}[Proof of Proposition~\ref{pr:eigen}]  For notational simplicity in the course of this proof, we will simply write $[\gamma]$ in place of $[\gamma]_{i+1}$.

We start with (\ref{it:eigen...ai.is.minus.infinity}).
    If $\tilde a_i=-\infty$ then by definition $a_j=-\infty$ for all $0\le
    j\le i$ and so $X$ is the $R_{i+1}G$-module $\langle [\alpha] \rangle$
    on which $\sigma$ acts as multiplication by $d$.  Hence $X$ is a
    $G$-eigenmodule.

    We now move on to (\ref{it:eigen...not.eigenmodule}).  Assume then that $\tilde a_i\neq -\infty$, so that $i^*\neq
    -\infty$ and $a_{i^*}\neq -\infty$, and assume that $i\ge 1$.

    We claim that $d^{p^{\tilde a_i}}\in U_{i+1}$.  If $d\in
    U_{i+1}$ then we are done, so assume that $d \not\in U_{i+1}$.  Recall that Lemma \ref{le:dinuomega} gives us $d \in U_\omega$, and hence $\omega \leq i$.  
    Since $d\not\in U_{i+1}$ we cannot be in the case $\omega = \nu$ (since otherwise Theorem \ref{th:d.is.cyclotomic.character}(\ref{it:d.is.cyclo.character...omega.equals.nu}) gives $d \in U_{i+1}$), so by Lemma \ref{le:dinuomega} we have $d \not\in U_{\omega+1}$.  Note additionally that since $\omega<\nu$, when $p=2$ we must have $\omega \geq 2$. In any case, then, Lemma~\ref{le:upowerIII} tells us  that $d^{p^{\tilde a_i}}\in U_{\omega+\tilde a_i}$.  If $i^* \geq \omega$, then Proposition \ref{pr:minimality.conditions}(\ref{it:exc.mod.indecom.condition...lower.bound.based.on.UIII}) gives $a_{i^*}>i^*-\omega$, and therefore  $\omega+\tilde a_i=\omega+a_{i^*}+i-i^*>i$.  Otherwise $i^*<\omega$, and so $\omega+\tilde a_i = \omega + a_{i^*} + i-i^*\ge i+1$. In either case, we conclude $d^{p^{\tilde a_i}}\in U_{i+1}$.


    Suppose that $X = \langle [\alpha],[\delta_0],[\delta_1],\dots,[\delta_{i}]\rangle$ is an $R_{i+1}H_{\tilde
    a_i}$-eigenmodule, and let $\tau=\sigma^{p^{\tilde a_i}}$.
    Then there exists $e\in R_{i+1}$ such that
    $(\tau-e)X=\{0\}$.

    We claim first that $e=1$.  Observe that
    \begin{equation*}
        (\tau-1)[\alpha]=(\sigma^{p^{\tilde a_i}}-d^{p^{\tilde a_i}})[\alpha]=
        Q_d(\tilde a_i,0)(\sigma-d)[\alpha] =Q_d(\tilde
        a_i,0)\sum_{j=0}^{i^*} p^j[\delta_j].
    \end{equation*}
    Let $X^0=\langle [\delta_0],\dots,[\delta_i]\rangle$.  We have shown that
    $(\tau-1)[\alpha]\in X^0$.  Then
    \begin{equation*}
        (\tau-1)[\alpha]-(\tau-e)[\alpha]=(e-1)[\alpha]\in X^0.
    \end{equation*}
    By Lemma~\ref{le:inter}, $e-1\in p^{\nu}R_{i+1}$.  Since $i<\nu$
    we obtain $e-1 = 0 \in R_{i+1}$.  Hence $e=1$.

   Now
    \begin{equation*}
        (\tau-1)[\alpha]=(\sigma^{p^{\tilde a_i}}-d^{p^{\tilde a_i}})[\alpha]=
        Q_d(\tilde a_i,0)(\sigma-d)[\alpha] =Q_d(\tilde
        a_i,0)\sum_{j=0}^{i^*} p^j[\delta_j] = 0.
    \end{equation*}
    Consider $0\le j\le i^*$ such that $a_j \neq -\infty$, and let $\chi_{a_j}:R_{i+1}G\to
    R_{i+1}G_{a_j}$ be the natural $R_{i+1}$-homomorphism. Note that if we are in the case $p=2$ then either $\omega = \nu$ or $\omega \geq 2$.  In the former case Theorem \ref{th:d.is.cyclotomic.character}(\ref{it:d.is.cyclo.character...omega.equals.nu}) gives $d \in U_{i+1} \subseteq U_2$, whereas in the latter case Lemma \ref{le:dinuomega} gives $d \in U_\omega \subseteq U_2$.  Hence we may apply Lemma~\ref{le:qhomoIII} and conclude
    \begin{equation*}
        \chi_{a_j}(Q_d(\tilde a_i,0)) = p^{\tilde a_i-a_j}Q_d(
        a_j,0) \mod p^{\tilde a_i-a_j+1}R_{i+1}G_{a_j}
    \end{equation*}
    Hence
    \begin{equation*}
        Q_d(\tilde a_i,0)\in p^{\tilde a_i-a_j}Q_d(a_j,0) + \langle
        p^{\tilde a_i-a_j+1}, (\sigma^{p^{a_j}}-1)\rangle.
    \end{equation*}
    Since $(\sigma^{p^{a_j}}-1)\delta_j=0$ by the defining relations (\ref{eq:definition.of.Xadm}), we have
    \begin{equation}\label{eq:isXaneigenmodule}
        (\tau-1)\alpha=
        \sum_{j=0}^{i^*} p^{\tilde a_i-a_j+j}Q_d(a_j,0)\delta_j+r_j\delta_j = 0.
    \end{equation}
    for some $r_j\in p^{\tilde a_i-a_j+j+1}R_mG$. (Note: this sum can include $j$ for which $a_j = -\infty$, since then $\delta_j = 0$ and therefore the corresponding summands are trivial.)

    Again appealing to (\ref{eq:definition.of.Xadm}) we deduce
    \begin{multline}\label{eq:eigene1}
        \sum_{j=0}^{i^*} p^{\tilde a_i-a_j+j}Q_d(a_j,0)[\delta_j]+r_j[\delta_j] = \\
        f_\alpha\left((\sigma-d)[\alpha]-\sum_{j=0}^{i^*} p^j[\delta_j]\right) + \sum_{j=0}^{i^*}
        f_j(\sigma^{p^{a_j}}-1)[\delta_j].
    \end{multline}

    Considering the coefficients of $[\alpha]$ in \eqref{eq:eigene1}, we have
    in $R_{i+1}G$
    \begin{equation*}
      0 = f_\alpha(\sigma-d).
    \end{equation*}
    By Lemma~\ref{le:phidbIII}, $f_\alpha=r_\alpha Q_d(n,0)$ for some $r_\alpha\in
    R_{i+1}G$.
    Now consider the coefficient of $[\delta_{i^*}]$ in \eqref{eq:eigene1}:
    \begin{equation*}
        p^{\tilde a_i-a_{i^*}+i^*}Q_d(a_{i^*},0)+r_{i^*} =
        -p^{i^*}f_\alpha + f_{i^*}(\sigma^{p^{a_{i^*}}}-1).
    \end{equation*}
    Substituting for $f_\alpha$, we have
    \begin{equation*}
        p^{\tilde a_i-a_{i^*}+i^*}Q_d(a_{i^*},0) + r_{i^*}=
        -p^{i^*}r_\alpha Q_d(n,0) + f_{i^*}(\sigma^{p^{a_{i^*}}}-1).
    \end{equation*}

    Let $\chi_{a_{i^*}}:R_{i+1} G\to R_{i+1} G_{a_{i^*}}$ be the
    natural $R_{i+1}$-homomorphism. Applying $\chi_{a_{i^*}}$ to the
    last equation and using Lemma~\ref{le:qhomoIII} again, we have in
    $R_{i+1}G_{a_{i^*}}$
    \begin{equation*}
        p^{\tilde a_i-a_{i^*}+i^*}
        Q_d(a_{i^*},0) + \chi_{a_{i^*}}(r_{i^*})=
        -p^{i^*}\chi_{a_{i^*}}(r_\alpha)p^{n-a_{i^*}}Q_d(a_{i^*},0).
    \end{equation*}
    Since $\tilde a_i \neq -\infty$ and $i<\nu$, we have $\tilde a_i<\tilde a_\nu$.  Proposition \ref{pr:aiversusnu} gives $\tilde a_\nu = n$, and so we have $\tilde
    a_i-a_{i^*}+i^*<n-a_{i^*}+i^*$.  Working modulo $p^{\tilde
    a_i-a_{i^*}+i^*+1}R_{i+1}G_{a_{i^*}}$, we have
    \begin{equation*}
        p^{\tilde a_i-a_{i^*}+i^*}Q_d(a_{i^*},0) = 0 \pmod{p^{\tilde
        a_i-a_{i^*}+i^*+1}R_{i+1}G_{a_{i^*}}}.
    \end{equation*}
    Recalling that $Q_d(a_{i^*},0)=P(a_{i^*},0)\pmod{p\Z
    G_{a_{i^*}}}$, we deduce
    \begin{equation*}
        p^{\tilde a_i-a_{i^*}+i^*}P(a_{i^*},0) = 0 \pmod{p^{\tilde
        a_i-a_{i^*}+i^*+1}R_{i+1}G_{a_{i^*}}}.
    \end{equation*}
    The constant term in $P(a_{i^*},0)$ is $1$, and so by comparing constant terms on both sides we have
      $p^{\tilde a_i-a_{i^*}+i^*}= 0 \in R_{i+1}$, whence $\tilde a_i-a_{i^*}+i^*\ge i+1$.
    Now by definition $\tilde a_i= a_{i^*}+(i-i^*)$, so
    \begin{equation*}
      i+1 \leq \tilde a_i - a_{i^*}+i^* = a_{i^*}+(i-i^*)-a_{i^*}+i^* = i,
    \end{equation*}
    a contradiction.
    Hence $X$ is not an $R_{i+1}H_{\tilde a_i}$-eigenmodule.

Now we turn to property $\mathcal{P}(H_{\tilde a_i})$.
    Taking a closer look, observe that we have shown that
    \begin{equation*}
        (\tau-1)[\alpha]= \sum_{j=0}^{i^*} p^{\tilde
        a_i-a_j+j}Q_d(a_j,0)[\delta_j] +r_j[\delta_j] \neq 0.
    \end{equation*}
    We claim that 
    $(\tau-1)X \subset pX$. 

    By Proposition \ref{pr:minimality.conditions}(\ref{it:exc.mod.indecom.condition...ai.vs.aj.inequalityIII}),
    $a_j<a_{i^*} \leq \tilde a_{i}$ for all $0\le j<i^*$. Because $(\sigma^{p^{a_j}}-1)\delta_j=0$ and
    \begin{equation*}
        (\sigma^{p^{a_j}}-1) P(\tilde a_{i},a_j) =
        (\sigma^{p^{\tilde a_{i}}}-1)=(\tau-1)
    \end{equation*}
    whenever $a_j \neq -\infty$, we deduce that $(\tau-1)$ annihilates $X^0=\langle
    [\delta_0],\dots,[\delta_i]\rangle$. Hence
    \begin{equation*}
        (\tau-1,p)X = \langle(\tau-1,p)\rangle[\alpha] +pX^0.
    \end{equation*}

    Proposition \ref{pr:minimality.conditions}(\ref{it:exc.mod.indecom.condition...ai.vs.aj.inequalityIII}) also gives $\tilde a_i-a_{i^*}+i^*<\tilde a_i-a_j+j$.  Hence
    \begin{equation*}
        (\tau-1)[\alpha]= p^{\tilde a_i-a_{i^*}+i^*}\left(
        Q_d(a_{i^*},0)[\delta_{i^*}] + p[z]\right) \neq 0.
    \end{equation*}
    for some $[z]\in X^0$.  Recalling that $i = \tilde a_i - a_{i^*} + i^*$, we have
    \begin{equation*}
        (\tau-1)X=(\tau-1) \langle [\alpha]\rangle\in p^{i}X \subseteq pX.
    \end{equation*}

    Let $$[w] = p^{-i}\sum_{j=0}^{i^*} p^{\tilde a_{i^*} - a_j + j}Q_d(a_{j},0)[\delta_{j}]+r_{j}[\delta_{j}].$$  We next claim that $[w],[\alpha]\not\in (\tau-1,p)X$.  Since $(\tau-1)X
    \subset pX$, it is enough to show that $[w],[\alpha]\not\in pX$.    
    Recall that $r_{j}\in p^{\tilde a_i-a_j+j+1}R_mG \subseteq pR_mG$ and  $Q_d(a_{i^*},0)=P(a_{i^*},0)\pmod{p\Z G}$.  Hence
    \begin{equation*}
        [w] = P(a_{i^*},0)[\delta_{i^*}] \pmod{pX}.
    \end{equation*}

    From the defining relations (\ref{eq:definition.of.Xadm}) we see that as an $\Fp$-module,
    $X/pX$ has an $\Fp$-base consisting of $\{[\alpha],\{\{\sigma^{j}[\delta_k]\}_{j=0}^{p^{a_k}-1}\}_{k=0}^{i}\}$.
    By definition of $P(a_{i^*},0)$, then, $[w]\not\in pX$.
    Hence $[w]\not\in (\tau-1,p)X$, and similarly $[\alpha]\not\in
    (\tau-1,p)X$.  Then we have solved the equation
    \begin{equation*}
        (\tau-1)[\alpha]=p^i[w]
    \end{equation*}
    with $[w],[\alpha] \not\in (\tau-1,p) X$.  Therefore $X$ satisfies property $\mathcal{P}(H_{\tilde a_i})$.

    Finally, we prove (\ref{it:eigen...is.trivial.module}). Let $H=H_{\tilde a_i+1}$, and for $\tau = \sigma^{p^{\tilde a_i+1}}$  note that $\langle \tau \rangle = H$.  We show that $X$
    is a trivial $R_{i+1}H$-module. 
    
    

	By definition, $\tilde
    a_i=a_{i^*}+(i-i^*)$. By Proposition \ref{pr:minimality.conditions}(\ref{it:exc.mod.indecom.condition...ai.vs.aj.inequalityIII}), $a_j\le a_{i^*}$ for all
    $0\le j\le i^*$.  Since $(\sigma^{p^{a_j}}-1)\delta_j=0$,
    we have that $\tau$ fixes every $[\delta_i]$ for $0\le j\le i$.  It will suffice to
    show that $[\alpha]$ is fixed by $\tau$ as well.

    Toward this end, in the same way that we established equation (\ref{eq:isXaneigenmodule}), one can show
    \begin{align*}
      (\tau-1)[\alpha] &= (\sigma^{p^{\tilde a_i+1}}-d^{p^{\tilde a_i+1}})[\alpha]
      = (\sigma-d)Q_d(\tilde a_i+1,0)[\alpha]  \\
      &= \sum_{j=0}^i p^jQ_d(\tilde
      a_i+1,0)[\delta_j] \\
      &= \sum_{j=0}^{i^*} p^{\tilde a_i-a_j+j+1}Q_d(a_j,0)[\delta_j]+u_j[\delta_j]
    \end{align*}
    for $u_j\in p^{\tilde a_i-a_j+j+2} R_{i+1}G$.
    If $0 \leq j <i^*$ satisfies $a_j \neq -\infty$, then by Proposition \ref{pr:minimality.conditions}(\ref{it:exc.mod.indecom.condition...ai.vs.aj.inequalityIII}) we have
    \begin{align*}
        \tilde
        a_i-a_j+j+1 &= a_{i^*}+(i-i^*)-a_j+j+1 \\
        &>(i-i^*)+(i^*-j)+j+1 \\ &=
        i+1.
    \end{align*}
    Hence
    \begin{align*}
        (\tau-1)[\alpha]=0
    \end{align*}
    and we have shown that $X$ is a trivial $R_{i+1}H$-module.
\end{proof}

\section{Proof of Theorem~{\ref{th:a.in.relation.to.norms}}}\label{se:proof1}


We first prove that if $m > i > \nu$, then $a_i = -\infty.$  This follows directly from Proposition \ref{pr:cyclopair} when $K/F$ is cyclotomic, so assume that $K/F$ is not cyclotomic.  Note that Proposition \ref{pr:aiversusnu} gives $a_{\nu^*} =n+\nu^*-\nu$. If $a_i \neq -\infty$, then Proposition \ref{pr:minimality.conditions}(\ref{it:exc.mod.indecom.condition...ai.vs.aj.inequalityIII}) gives $$a_i > i-\nu^*+a_{\nu^*} =i+n-\nu > n.$$  This is a clear contradiction.

Recall that for $0\le i<\nu$ we define 
\begin{equation*}
    b_i =
    \begin{cases}
        -\infty, & \xi_{p^{i+1}}\in N_{K/F(\xi_{p^{i+1}})}(K^\times)
        \\
        \min \{s : \xi_{p^{i+1}}\in N_{K/K_{s+1}}(K^\times)\}, &
        \text{\ otherwise}
    \end{cases}
\end{equation*}
and set $b_\nu=n$.  Our strategy will be to prove that for all $0 \leq i \leq \min\{\nu,m-1\}$ we have $\tilde a_i = b_i$.  Theorem \ref{th:a.in.relation.to.norms} for $0 \leq i \leq \min\{\nu,m-1\}$ then follows by applying Proposition \ref{pr:interp}.  

	We first settle the cases where $\nu = \infty$ and when $K/F$ is cyclotomic.  If $\nu = \infty$, then since we are not in the case $p=2$ and $\omega = 1 < \nu$, we also deduce that $\omega =\infty$.  Hence for all $i$ we have $F(\xi_{p^{i+1}}) = F$, and moreover $\xi_{p^{i+1}} = \left(\xi_{p^{i+n+1}}\right)^{p^n} = N_{K/F}\left(\xi_{p^{i+n+1}}\right)$ for an appropriately chosen $p^{n+i+1}$st root of unity $\xi_{p^{i+n+1}}$.  Hence $b_i = -\infty$ for all $0 \leq i \leq m$.  We have already seen that $a_i = -\infty$ for all $0 \leq i \leq m$ in Proposition \ref{pr:nuinfty}, and therefore $\tilde a_i = -\infty$ for all $0 \leq i \leq m$ as well.
	
	Now suppose that $K/F$ is cyclotomic.  Since we are not in the case $p=2$ and $\omega = 1 < \nu$, it must be the case that $\omega + n = \nu$, and moreover that $\omega \geq 2$ when $p=2$.  Lemma \ref{le:norm.of.root.of.unity} gives us $\xi_{p^{i+1}} \in N_{K/F(\xi_{p^{i+1}})}(K^\times)$, so that $b_i = -\infty$ for all $0 \leq i \leq \min\{\nu,m\}$. By Proposition \ref{pr:cyclopair} we have $a_i = -\infty$ for all $0 \leq i \leq \min\{\nu,m\}$, and hence $\tilde a_i = -\infty$ for all $0 \leq i \leq \min\{\nu,m\}$ as well.  

    So suppose $K/F$ is not cyclotomic and that $\nu\neq \infty$, and let $0 \leq i \leq \nu$.  By Lemma \ref{le:contracting.and.d.equivalence} we may assume without
    loss of generality that $m=i+1$. First consider $i=0$.  Then $\tilde a_0=a_0$ and the equality
    $\tilde a_0=b_0$ is \cite[Theorem 3]{MSS}.  In the case $i=\nu$, Proposition \ref{pr:aiversusnu} gives $a_{\nu^*} = n+\nu^*-\nu$, and so $$\tilde a_\nu = a_{\nu^*} + (\nu-\nu^*) = n = b_\nu.$$  
    
    Now consider $i$ with $1\le i<\nu$.  We will again denote an element $[\gamma]_m \in J_m$ by $[\gamma]$ for the course of this proof.   Let $(\alpha,\{\delta_j\}_{j=0}^{i+1})$ represent $(\mathbf{a},d)$, and let $X = \langle [\alpha],[\delta_0],[\delta_1],\cdots,[\delta_{i}]\rangle$. Recall that the relations satisfied by $X$ are generated by the equations in (\ref{eq:definition.of.Xadm}).  Write
    $H_l=\Gal(K/K_l)$ for $0\le l\le n$.


    We claim first that $b_i\le \tilde a_i$.

    Suppose first that $\tilde a_i=-\infty$.  Then $a_j=-\infty$ for all $0\le j\le i$, and from the defining relations (\ref{eq:definition.of.Xadm}) we have that $X$ is a cyclic
    $R_mG$-eigenmodule with eigenhomomorphism given by $\phi_d^{(i+1)}$. By
    \cite[Th.~1]{MSS2b}, $J_{i+1}$ decomposes into indecomposables as
    the sum of $X$ and cyclic $R_{i+1}G$-submodules free as
    $R_{i+1}(G/H_l)$-modules for $H_l\le G$.  Let $t$ be defined by
    $K_t=F(\xi_{p^{i+1}})$ and let $H=H_t$. Observe that $i+1\le
    t+\omega$. Now then $X$ is a cyclic $R_{i+1}H$-eigenmodule with
    eigenhomomorphism $\phi_d^{(i+1)}\vert_H$. Furthermore, any free $R_{i+1}(G/H_l)$-module decomposes into a sum of free $R_{i+1}(G/H\cap H_l)$-modules when considered as an $R_{i+1}H$-module.  Hence as an $R_{i+1}H$-module,
    $J_{i+1}$ is the sum of a cyclic $R_{i+1}H$-eigenmodule with
    eigenhomomorphism $\phi_d^{({i+1})}\vert_H$ and cyclic $R_{i+1}H$-modules free
    as $R_{i+1}(H/(H\cap H_l))$-modules for $H_l\le G$.

    Now suppose instead that $\tilde a_i\neq -\infty$. Let $t=\tilde
    a_i + 1$, and $H=H_t$.  By \cite[Th.~1]{MSS2b}, $J_{i+1}$
    decomposes into the sum of $X$ and cyclic $R_{i+1}G$-submodules free
    as $R_{i+1}(G/H_l)$-modules for $H_l\le G$. Moreover, from Proposition \ref{pr:eigen}(\ref{it:eigen...is.trivial.module}) we
    have that $X$ is a trivial $R_{i+1}H$-module, hence a direct sum of
    trivial cyclic $R_{i+1}H$-modules. 
    Hence as an $R_{i+1}H$-module, $J_{i+1}$
    decomposes into indecomposables as the sum of trivial
    cyclic $R_{i+1}H$-modules and cyclic $R_{i+1}H$-submodules free as
    $R_{i+1}(H/(H\cap H_l))$-modules for $H_l\le G$.

     By \cite[Th.~1]{MSS2b} for the
    extension $K/K_{t}$, $J_{i+1}$ --- considered as an $R_{i+1}H$-module ---
    contains as an indecomposable summand an $(i+1)$-exceptional module
    $\hat X$ defined by $$\hat X = \langle [\hat \alpha],[\hat \delta_0],[\hat \delta_1],\cdots,[\hat \delta_{i}]\rangle$$ where $(\hat\alpha,\{\hat\delta_j\})$ represents a minimal norm pair $(\hat{\mathbf{a}},\hat d)$ of length $i+1$  for $K/K_t$.  By the Krull-Schmidt-Azumaya Theorem (see \cite[Theorem~12.6]{AnFu}), the decomposition of $J_{i+1}$ as a
    $R_{i+1}H$-module is unique. Hence this
    indecomposable summand is one of the ones considered in the last
    two paragraphs.  In particular, $\hat X$ is a cyclic
    $R_{i+1}H$-module which is free as an $R_{i+1}(H/(H \cap H_l))$-module for some $H_l \leq G$. By the defining relations (\ref{eq:definition.of.Xadm}) for $\hat X$ (with $\tau$ in place of $\sigma$), the cyclicity of $\hat X$ implies $\hat X = \langle [\hat \alpha]\rangle$, and so $\hat \delta_j = 0$ for all $1 \leq j \leq i$.  
    
    We claim that additionally $\delta_0 =0$.  If $p>2$ this follows because otherwise $\hat X/p\hat X$ would be an $\F_pH$-module of dimension $p^{\hat a_0}+1$, an impossibility if $\hat X$ is free as an $R_{i+1}(H/(H \cap H_l))$-module. This same argument holds when $p=2$ unless $a_0 = 0$ and $H_l$ is index $2$ in $H$, so suppose these equalities hold.  Observe that in that case we have $(\tau-1)[\hat \alpha] + (1-\hat d)[\hat \alpha] = (\tau - \hat d)[\hat \alpha] = [\delta_0]$.  After applying $(\tau-1)$ to this relation and using the fact that $(\tau-1)^2 = -2(\tau-1) \in R_{i+1}(H/H\cap H_l)$, we recover $$-2(\tau-1) [\hat \alpha] + (1-\hat d)(\tau-1) [\hat \alpha] = (\tau-1)[\delta_0] = [0].$$ By the freeness of $\langle[\alpha]\rangle$ as an $R_{i+1}(H/H\cap H_l)$-module, we conclude that if $p=2$, $a_0 = 0$ and $H_l$ is index $2$ in $H$, then $-2 \equiv 1- \hat d \pmod{2^{i+1}}$.  In other words, in this case we have $\hat d = 1 + 2 + 2^{i+1}x$ for some $x \in \Z$.  We now show that this provides a contradiction, by establishing that in all cases (i.e., not only when $p=2$, $a_0=0$ and $H_l$ is index $2$ in $H$) we have $\hat d \in U_{i+1}$.  Once this is established, we will have shown that $\delta_j = 0$ for all $0 \leq j \leq i$.
    
    Suppose, for the sake of contradiction, that $\hat d\not\in U_{i+1}$.  If $\mu$ is chosen so that $\xi_{{p^{\mu}}} \in K_{t}$ but $\xi_{p^{\mu+1}}\not\in K_t$, then Lemma \ref{le:dinuomega} gives $\hat d\in U_{\min\{\mu,i+1\}}$.  By assumption, we must be in the case where $\mu < i+1$.  Note that because $\omega \leq \mu$, we further have $\xi_{p^{i+1}} \in K_{i+1-\omega}$.  Furthemore we must have $\mu<\nu$, since otherwise Theorem \ref{th:d.is.cyclotomic.character}(\ref{it:d.is.cyclo.character...omega.equals.nu}) gives $\hat d \equiv 1 \pmod{p^{i+1}}$. 

    Now if $\tilde a_i=-\infty$ then $K_t$ is defined as $F(\xi_{p^{i+1}})$.  By the definition of $\mu$ we have $\mu \geq i+1$, which is a contradiction.  So assume that $\tilde a_i\neq -\infty$,  and hence $\tilde a_i=a_{i^*}+(i-i^*)$.  Recall that
    $t=\tilde a_i+1$.  If $i^*< \omega$ then
    \begin{equation*}
        t = \tilde a_i+1=a_{i^*}+(i-i^*)+1>
        i+1-\omega.
    \end{equation*}
    This gives $\xi_{p^{i+1}} \in K_{i+1-\omega} \subseteq K_t$, and hence we have $\mu \geq i+1$, a contradiction. 
    
    Instead suppose
    $i^*\ge \omega$.  Since $\omega \leq \mu < \nu$ we have $d \in U_\omega \setminus U_{\omega+1}$ by Lemma \ref{le:dinuomega}. Then by Proposition \ref{pr:minimality.conditions}(\ref{it:exc.mod.indecom.condition...lower.bound.based.on.UIII}) we have
    $a_{i^*}> i^*-\omega$.  Therefore
    \begin{align*}
       t=\tilde a_i+1 = a_{i^*}+(i-i^*)+1 > i^*-\omega+(i-i^*)+1 \ge i+1 - \omega.
    \end{align*}
    As in the previous case, this provides a contradiction.  Hence it must be the case that $\hat d \in U_{i+1}$ in all cases, as desired.


    At this point we have shown  --- in all cases --- that $((-\infty,\cdots,-\infty),\hat d)$ is a minimal norm pair for $\hat d \in U_{i+1}$, with $(\hat \alpha,\{\hat\delta_j\})$ representing this norm pair. By definition we have
    \begin{equation*}
      \xi_{p} = \frac{\hat d-1}{p}N_{K/K_t}(\hat\alpha) +
      p^iN_{K/K_t}(\hat\delta_{i+1})
    \end{equation*}
    Since $\hat d\in U_{i+1}$, we may take $p^i$th roots in
    $K_{t}^\times$ to obtain that $\xi_{p^{i+1}}\in
    N_{K/K_{t}}(K^\times)$.  This gives $b_i<t$.  If $\tilde a_i\neq -\infty$ then $t=\tilde
    a_i+1$ and so $b_i\le \tilde a_i$.  If $\tilde a_i=-\infty$ then
    $K_t=F(\xi_{p^{i+1}})$, and so if $b_i<t$ then $b_i=-\infty$, by
    definition.  Hence $b_i\le \tilde a_i$ in all cases.
    
    Now we claim that $\tilde a_i\le b_i$.  
%
    For the sake of contradiction, suppose that $b_i<\tilde a_i$.

    First, suppose $b_i=-\infty$.  Let $t$ be defined as
    $K_t=F(\xi_{p^{i+1}})$ and set $H=H_t$.  Since $K/F$ is not
    cyclotomic, $t<n$. By the definition of $b_i$ we have $\xi_{p^{i+1}}\in N_{K/K_t}(K^\times)$,
    say $\xi_{p^{i+1}}=N_{K/K_t}(\hat \beta)$. Raising to $p^{i+1}$st
    powers and using Hilbert 90, we obtain $\hat \alpha$ with
    $(\sigma^{p^t}-1)\hat \alpha =p^{i+1}\hat \beta$.  Hence a minimal norm
    pair of length $i+1$ for the extension $K/K_t$ is
    $((-\infty,\dots,-\infty),1)$, and so the exceptional module $\hat X = \langle \hat \alpha \rangle$ for $J_{i+1}$ --- considered as an $R_{i+1}H$-module --- is a trivial cyclic module.  Moreover, by Lemma~\ref{le:inter} the generator
    $\hat \alpha$ is annihilated only by $p^\nu$, and since
    $m=i+1\le \nu$, the module is isomorphic to $R_m$.

    By \cite[{Th.~1}]{MSS2b}, $J_m$ --- considered as an $R_mH$-module --- contains
    as indecomposables $\hat X$ and cyclic $R_mH$-modules which are
    free as $R_m(H/S)$-modules for $S\le H$.  
    Now since $X$ is a direct summand of $J_m$ (considered as an $R_mG$-module), $X$ must decompose as an $R_mH$-module
    into a direct sum of these indecomposable $R_m H$-modules, so that $X$ is then a direct sum of free $R_m(H/S)$-modules for $S \leq H$.  Corollary \ref{cor:direct.sums.do.not.have.property.P} tells us that $X$ does not satisfy property $\mathcal{P}(H)$.

	Note that our assumption $b_i < \tilde a_i$ implies $\tilde a_i \geq 0$ in this case.  Hence by 
	Proposition~\ref{pr:eigen} and the comment following the definition of $\mathcal{P}(H)$, we have 
    $0 < \tilde a_i +1 \le t$. Observe that $t=i+1-\omega$, and furthermore that $t>0$ implies $\omega < i+1 \leq \nu$.  Since we know $\tilde a_i + 1 \leq t$, we have
    $\tilde a_i\le t-1=i-\omega$. Since 
    $\tilde a_i =a_{i^*}+(i-i^*)$,  we deduce
    that $a_{i^*}\le i^*-\omega$. Since $a_{i^*}\neq -\infty$, this gives
    $i^*\ge \omega$. Now because $\omega<\nu$, Lemma \ref{le:dinuomega} gives $d \in U_\omega \setminus U_{\omega +1}$, and then Proposition \ref{pr:minimality.conditions}(\ref{it:exc.mod.indecom.condition...lower.bound.based.on.UIII}) gives $a_{i^*}> i^*-\omega$, a contradiction.  

    Now consider $b_i\ge 0$, and let $t=\tilde a_i\ge b_i+1$ and
    $H=H_t$. Then $\xi_{p^{i+1}}\in N_{K/K_{t}}(K^\times)$, say
    $\xi_{p^{i+1}}=N_{K/K_{t}}\hat\delta$.  Raising both sides to the $p^{i}$th power gives
    \begin{equation*}
        \xi_p=p^{i}N_{K/K_{t}}\hat\delta.
    \end{equation*}
    Let $\tau=\sigma^{p^{t}}$. Taking $p$th powers and
    applying Hilbert's Theorem 90, there exists an element
    $\hat\alpha\in K^\times$ with $(\tau-1)\hat\alpha=
    p^{i+1}\hat\delta$.  Therefore $((-\infty,\dots,-\infty),1)$ is
    a norm pair (and therefore the minimal norm pair) of length $i+1$.

    By \cite[Th.~1]{MSS2b}, the $(i+1)$-exceptional module $\hat
    X$ for $K/K_{t}$ is a trivial cyclic $R_mH$-module.  By
    Lemma~\ref{le:inter}, since
    $m=i+1\le \nu$, the module is isomorphic to $R_m$. By
    Krull-Schmidt-Azumaya, each indecomposable summand of the
    $R_mH$-module $J_m$ in the decomposition of
    Theorem~1 is either $\hat X$ or a free
    $R_m(H/S)$-module for some $S\le H$. Hence $X$, as a direct
    summand of $J_m$, decomposes into a direct sum of such
    $R_mH$-indecomposables, and by Corollary \ref{cor:direct.sums.do.not.have.property.P} it follows that $X$ does not satisfy Property
    $\mathcal{P}(H)$. 
    This contradicts Proposition~\ref{pr:eigen}.


\begin{thebibliography}{MSS}

\bibitem{Al} {\sc A.~Albert}.  On cyclic fields.
\textit{Trans.~Amer.~Math.~Soc.} \textbf{37} (1935), no.~3,
454--462.

\bibitem{AnFu} {\sc F.~Anderson}, {\sc K.~Fuller}. \emph{Rings
and categories of modules}. Graduate Texts in Mathematics 13. New
York: Springer-Verlag, 1973.

\bibitem{ArasonFeinSchacherSonn} {\sc J.Kr.~Arason}, {\sc B.~Fein}, {\sc M.~Schacher}, {\sc J.~Sonn}. Cyclic extensions of $K(\sqrt{-1})/K$. \emph{Trans.~Amer.~Math.~Soc.} {\bf 313} (1989), no.~2, 843--851.

\bibitem{BergSchultz} {\sc J.~Berg}, {\sc A.~Schultz}.  $p$-groups have unbounded realization multiplicity.  \emph{Proc.~AMS}. {\bf 142} (2014), no.~7, 2281--2290.

\bibitem{BertrandiasPayan} {\sc F.~Bertrandias}, {\sc J.-J.~Payan}. $\Gamma$-extensions et invariants cyclotomiques. \emph{Ann.~Sci.~\'Ec.~Norm.~Sup\'er.~(4)}. {\bf 5} (1972), no.~4, 517--543.

\bibitem{BLMS} {\sc G.~Bhandari}, {\sc N.~Lemire}, {\sc J.~Min\'a\v{c}}, {\sc J.~Swallow}. Galois module structure of Milnor $K$-theory in characteristic $p$. \textit{New York J.~Math.} \textbf{14} (2008), 215--224.

\bibitem{Borevic} {\sc Z.I.~Borevi\v{c}}. The multiplicative group of cyclic $p$-extensions of a local field. (Russian) \textit{Trudy Mat. Inst. Steklov} \textbf{80} (1965), 16--29. English translation in \textit{Proc.~Steklov Inst.~Math.~No.~80 (1965): Algebraic number theory and representations}, edited by D.~K.~Faddeev, 15--30. Providence, RI: American Mathematical Society, 1968.

\bibitem{ByottChildsElder} {\sc N.P.~Byott}, {\sc L.N.~Childs}, {\sc G.G.~Elder}. Scaffolds and generalized integral Galois module structure. \emph{Ann.~Inst.~Fourier} {\bf 68} (2013), no.~3, 965--1010. 

\bibitem{ByottElder} {\sc N.P.~Byott}, {\sc G.G.~Elder}. Galois scaffolds and Galois module structure in extensions of characteristic $p$ local fields of degree $p^2$. \emph{J.~Number Theory} {\bf 133} (2013), 3598--3610.

\bibitem{CMSHp3} {\sc S.~Chebolu}, {\sc J.~Min\'{a}\v{c}}, {\sc A.~Schultz}. Galois $p$-groups and Galois modules.  \emph{Rocky Mtn.~J.~Math.}  {\bf 46} (2016), 1405--1446.

\bibitem{CMSS} {\sc F.~Chemotti}, {\sc J.~Min\'a\v{c}}, {\sc A.~Schultz}, {\sc J.~Swallow}. Galois module structure of square power classes for biquadratic extensions.  To appear in \emph{Canad.~J.~Math.}.  

\bibitem{Childs} {\sc L.N.~Childs}. \emph{Taming wild extensions: Hopf algebras and local Galois module theory}, Mathematical Surveys and Monographs, Vol.~80. Providence, RI: American Mathematical Society, 2000.

\bibitem{CoatesFukayaKatoSujatha} {\sc J.~Coates}, {\sc T. Fukaya}, {\sc K.~Kato}, {\sc R.~Sujatha}. Root numbers, Selmer groups and non-commutative Iwasawa theory. \emph{J. Alg. Geom.} {\bf 19}, (2010), 19–97.

\bibitem{CoatesSchneiderSujatha} {\sc J.~Coates}, {\sc P.~Schneider}, {\sc R.~Sujatha}. Modules over Iwasawa algebras. \emph{J.~Inst.~Math.~Jussieu} {\bf 2} (2003), no. 1, 73–108.

\bibitem{CoatesSujathaWintenberger} {\sc J.~Coates}, {\sc R.~Sujatha}, {\sc J.-P.~Wintenberger}. On the Euler-Poincar\'{e} characteristics of
finite dimensional $p$-adic Galois representations.  \emph{Publ.~Math.~Inst.~Hautes \'{E}tudes Sci.} {\bf 93} (2001), 107--143.

\bibitem{EfratMatzri} {\sc I.~Efrat}, {\sc E. Matzri}. Triple Massey products and absolute Galois groups. \textit{J.~Eur.~Math.~Soc.} {\bf 19} (2017), 3629-3640.

\bibitem{DameyMartinet} {\sc P.~Damey}, {\sc J.~Martinet}. Plongement d'une extension quadratique dans une extension quaternionienne. \emph{J.~Reine Angew.~Math.}{\bf 262/263} (1973), 323--338.

\bibitem{DameyPayan} {\sc P.~Damey}, {\sc J.-J.~Payan}. Existence et construction des extensions Galoisiennes et non-ab\'{e}liennes de degr\'{e} $8$ d'un corps de caract\'{e}ristique diff\'{e}rente de $2$. \emph{J.~Reine Angew.~Math.} {\bf 244} (1970), 37--54.


\bibitem{Dedekind} {\sc R.~Dedekind}. Konstruktion von Quaternionk\"{o}rpern. Gasammelte Mathematische Werke, Band 2, Vieweg, Braunschweig, 1931, 376--384.




\bibitem{Faddeev} {\sc D.K.~Faddeev}. On the structure of the reduced multiplicative group of a cyclic extension of a local field. \textit{Izv.~Akad.~Nauk~SSSR~Ser.~Mat.} \textbf{24} (1960), 145--152.


\bibitem{FeinSaltmanSchacher} {\sc B.~Fein}, {\sc D.~Saltman}, {\sc M.~Schacher}. Heights of cyclic field extensions. \emph{Bull.~Soc.~Math.~Belg., Ser.~A}. {\bf 40} (1988), 213--223.

\bibitem{GrundmanSmith1} {\sc H.~Grundman}, {\sc T.~Smith}. Automatic realizability of Galois groups of order $16$. \emph{Proc.~Amer.~Math.~Soc.} {\bf 124} (1996), 2631--2640.

\bibitem{GrundmanSmith2} {\sc H.~Grundman}, {\sc T.~Smith}. Realizability and automatic realizability of Galois groups of order $32$. \emph{Cent.~Eur.~J.~Math.} {\bf 8} (2010), no.~2, 244-260.

\bibitem{GrundmanSmith3} {\sc H.~Grundman}, {\sc T.~Smith}. Galois realizability of groups of order $64$. \emph{Cent.~Eur.~J.~Math.} {\bf 8} (2010), no.~5, 846--854.

\bibitem{GrundmanSmithSwallow} {\sc H.~Grundman}, {\sc T.~Smith}, {\sc J.~Swallow}. Groups of order $16$ as Galois groups. \emph{Expo.~Math.} {\bf 13} (1995), 289--319.

\bibitem{HarpazWittenberg} {\sc Y.~Harpaz}, {\sc O.~Wittenberg}. The Massey vanishing conjecture for number fields.  Preprint available at \url{https://arxiv.org/pdf/1904.06512.pdf}.

\bibitem{HMS} {\sc L.~Heller}, {\sc J.~Min\'a\v{c}}, {\sc T.T.~Nguyen}, {\sc A.~Schultz}, {\sc N.D.~Tan}.  Galois module structure of some elementary $p$-abelian extensions.  Available on the arxiv at \url{https://arxiv.org/abs/2203.02604}.

\bibitem{IwasawaAnnals} {\sc K.~Iwasawa}. On $\mathbf{Z}_l$-extensions of algebraic number fields. \emph{Ann.~Math.} {\bf 98}, (1973), no.~2, 246--326.

\bibitem{Jannsen} {\sc U.~Jannsen}. Iwasawa modules up to isomorphism. \emph{Adv.~Stud.~Pure Math.} {\bf 17}, (1989), 171--207.

\bibitem{Jensen1} {\sc C.U.~Jensen}. On the representations of a group as a Galois group over an arbitrary field. \emph{Th\'{e}orie des nombres (Quebec, PQ, 1987)}, 441--458. Berlin: de Gruyter, 1989.

\bibitem{Jensen2} {\sc C.U.~Jensen}. Finite groups as Galois groups over arbitrary fields. \emph{Proceedings of the International Conference on Algebra, Part 2 (Novosibirsk, 1989)}, 435--448. \emph{Contemp. Math.} {\bf 131}, Part 2. Providence, RI: American Mathematical Society, 1992.

\bibitem{Jensen3} {\sc C.U.~Jensen}. Elementary questions in Galois theory. \emph{Advances in algebra and model theory (Essen, 1994; Dresden, 1995)}, 11--24. \emph{Algebra Logic Appl.} {\bf 9}. Amsterdam: Gordon and Breach, 1997.

\bibitem{JensenPrestel1} {\sc C.U.~Jensen}, {\sc A.~Prestel}. Unique realizability of finite abelian 2-groups as Galois groups. \emph{J.~Number Theory} {\bf 40} (1992), no. 1, 12--31.

\bibitem{JensenPrestel2} {\sc C.U.~Jensen}, {\sc A.~Prestel}. How often can a finite group be realized as a Galois group over a field?
\emph{Manuscripta Math.} {\bf 99} (1999), 223--247.

\bibitem{JensenLedetYui} {\sc C.U.~Jensen}, {A.~Ledet}, {N.~Yui}. \emph{Generic polynomials: constructive aspects of the inverse Galois problem}. Mathematical Sciences Research Institute Publications 45. Cambridge: Cambridge University Press, 2002.

\bibitem{Kato} {\sc K.~Kato}. Lectures on the approach to Iwasawa theory for Hasse-Weil $L$-functions via $B_{dR}$, Part I. In \emph{Arithmetic algebraic geometry}, Lecture Notes in Mathematics, Vol.~1553, editor {\sc E.~Ballico}, 50--163.  Berlin: Springer-Verlag, 1993.

\bibitem{LamLiuSharifiWakeWang} {\sc Y.H.J.~Lam}, {\sc Y.~Liu}, {\sc R.~Sharifi}, {\sc P.~Wake}, {\sc J.~Wang}. Generalized Bockstein maps and Massey products.  Prepring available at \url{https://arxiv.org/pdf/2004.11510.pdf}.

\bibitem{Ledet} {\sc A.~Ledet}. \emph{Brauer type embedding problems}. Providence, RI: American Mathematical Society (2005).

\bibitem{LMSS}  {\sc N.~Lemire}, {\sc J.~Min\'a\v{c}}, {\sc A.~Schultz}, {\sc J.~Swallow}. Galois module structure of Galois cohomology for embeddable cyclic extensions of degree $p^n$. \emph{J.~London Math.~Soc.} {\bf 81} (2010), no.~3, 525--543.

\bibitem{LMS} {\sc N.~Lemire}, {\sc J.~Min\'a\v{c}}, {\sc J.~Swallow}. Galois module structure of Galois cohomology and partial Euler-Poincar\'{e} characteristics. \textit{J. Reine Angew. Math.} {\bf 613} (2007), 147--173.

\bibitem{MassyNguyen} {\sc R.~Massy}, {\sc T.~Nguyen-Quang-Do}. Plongement d'une extension de degr\'{e} $p^2$ dans une surextension non ab\'{e}lienne de degr\'{e} $p^3$: \'{e}tude locale-globale. \emph{J.~Reine Angew.~Math.} {\bf 291}(1977), 149--161.

\bibitem{Michailov1} {\sc I.~Michailov}. Four non-abelian groups of order $p^4$ as Galois groups. \emph{J.~Alg.} {\bf 307} (2007), 287--299.

\bibitem{Michailov2} {\sc I.~Michailov}. Groups of order $32$ as Galois groups. \emph{Serdica Math.~J.} {\bf 33} (2007), no.~1, 1--34.

\bibitem{Michailov3} {\sc I.~Michailov}. Induced orthogonal representations of Galois groups. \emph{J.~Alg.} {\bf 322} (2009), 3713--3732.

\bibitem{Michailov4} {\sc I.~Michailov}.  On Galois cohomology and realizability of $2$-groups as Galois groups. \emph{Cent.~Eur.~J.~Math.} {\bf 9} (2011), no.~2, 403--419.

\bibitem{Michailov5} {\sc I.~Michailov}.  On Galois cohomology and realizability of $2$-groups as Galois groups II. \emph{Cent.~Eur.~J.~Math.} {\bf 9} (2011), no.~6, 1333--1343.

\bibitem{MPQT} {\sc J.~Min\'a\v{c}}, {\sc F.~Pasini},  {\sc C.~Quadrelli}, {\sc N.D.~T\^{a}n}. Koszul algebras and quadratic duals in Galois cohomology. To appear in \emph{Adv.~Math.} Available at \url{https://arxiv.org/pdf/1808.01695}.

\bibitem{MSS} {\sc J.~Min\'a\v{c}}, {\sc A.~Schultz}, {\sc J.~Swallow}. Galois module structure of $p$th-power classes of cyclic extensions of degree $p^n$.  \emph{Proc.~London Math.~Soc.} \textbf{96} (2006), 307--341.

\bibitem{MSS.auto} {\sc J.~Min\'a\v{c}}, {\sc A.~Schultz}, {\sc J.~Swallow}. Automatic realizations of Galois groups with cyclic quotient of order $p^n$.  \emph{J.~Th\'{e}or.~Nombres Bordeaux} {\bf 20} (2008), 419--430.

\bibitem{MSS2a} {\sc J.~Min\'a\v{c}}, {\sc A.~Schultz}, {\sc J.~Swallow}. On the indecomposability of a remarkable new family of modules appearing in Galois theory. \emph{J.~Alg.} {\bf 598} (2022), 194--235.

\bibitem{MSS2b} {\sc J.~Min\'a\v{c}}, {\sc A.~Schultz}, {\sc J.~Swallow}. Galois module structure of the units modulo $p^m$ of cyclic extensions of degree $p^n$. To appear in \emph{Manuscripta Math.}.  Available at \url{https://doi.org/10.1007/s00229-022-01385-z}.  

\bibitem{MS} {\sc J.~Min\'a\v{c}}, {\sc J.~Swallow}.  Galois module structure of $p$th-power classes of extensions of degree $p$. \textit{Israel J.~Math.} {\bf 138} (2003), 29--42.

\bibitem{MS2} {\sc J.~Min\'a\v{c}}, {\sc J.~Swallow}.  Galois embedding problems with cyclic quotient of order $p$. \textit{Israel J.~Math.}, {\bf 145} (2005), 93--112.

\bibitem{MT-Advances} {\sc J.~Min\'a\v{c}}, {\sc N.D.~T\^{a}n}. Construction of unipotent Galois extensions and Massey products. \textit{Adv.~Math.} {\bf 304} (2017), 1021--1054.

\bibitem{MT-JAlg} {\sc J.~Min\'a\v{c}}, {\sc N.D.~T\^{a}n}. Counting Galois $\mathbb{U}_4(\mathbb{F}_p)$-extensions using Massey products. \emph{J.~Number Theory} {\bf 176} (2017), 76--112.

\bibitem{MT-JEMS} {\sc J.~Min\'a\v{c}}, {\sc N.D.~T\^{a}n}. Triple Massey products and Galois theory. \textit{J.~Eur.~Math.~Soc.} {\bf 19} (2017), no.~1, 255--284.

\bibitem{Noether-collected} {\sc E.~Noether}. \emph{Gasammelte Abhandlungen - Collected Papers}. Berlin: Springer-Verlag, 1983.

\bibitem{Noether-normalbasis} {\sc E.~Noether}. Normalbasis bei K\"{o}rpern ohne h\"{o}here Verzweigung. \emph{J.~Reine Angew.~Math.} {\bf 167} (1932), 147--152.

\bibitem{R} {\sc K.~Rosenbaum}.  On the structure of $p$-adic closure of cyclic extensions of local fields.  \emph{Indian J. of Math.} \textbf{20} (1978), no.~3, 255--264.

\bibitem{Schultz} {\sc A.~Schultz}. Parameterizing solutions to any Galois embedding problem over $\mathbb{Z}/p^n\mathbb{Z}$ with elementary $p$-abelian kernel. \textit{J.~Alg.} {\bf 411} (2014), 50--91.

\bibitem{SharifiCrelle} {\sc R.~Sharifi}. Massey products and ideal class groups. \emph{J.~Reine Angew.~Math.} {\bf 603} (2007), 1--33.

\bibitem{SharifiNotices} {\sc R.~Sharifi}. Iwasawa theory: a climb up the tower. \emph{Notices Amer.~Math.~Soc.} {\bf 66} (2019), no.~1, 16--26.

\bibitem{SujathaEulerPoincare} {\sc R.~Sujatha}. Euler-Poincar\'{e} characteristics of $p$-adic Lie groups and arithmetic. \emph{Proceedings of the International Conference on Algebra, Arithmetic and Geometry}, TIFR, Bombay, (2000).

\bibitem{Swallow} {\sc J.~Swallow}. Central $p$-extensions of $(p,p,\cdots,p)$-type Galois groups. \emph{J.~Alg.} {\bf 186}(1996), 277--298.

\bibitem{Waterhouse} {\sc W.~Waterhouse}. The normal closures of certain Kummer extensions. \textit{Canad.~Math.~Bull.} {\bf 37} (1994), no.~1,133--139.

\bibitem{Whaples} {\sc G.~Whaples}. Algebraic extensions of arbitrary fields. \emph{Duke Math.~J.} {\bf 24} (1957), no.~2, 201--204.
\end{thebibliography}
\end{document}